\documentclass{amsart}

\usepackage{anysize}
\marginsize{3cm}{3cm}{3cm}{3cm}

\usepackage{amssymb,latexsym}
\usepackage{amsmath,amsthm}
\usepackage{amsfonts,mathrsfs}
\usepackage{mathtools}

\usepackage{tikz-cd}

\usetikzlibrary{decorations.pathmorphing}   
\tikzcdset{arrow style=math font}  

\usepackage{todonotes}

\usepackage[all]{xy}
\usepackage{graphicx}

\usepackage{xcolor,url}
\usepackage{hyperref}
\hypersetup{
    colorlinks,
    citecolor=blue,
    filecolor=blue,
    linkcolor=black,
    urlcolor=blue
}

\usepackage[shortlabels]{enumitem}

\theoremstyle{plain}
\newtheorem{theorem}{Theorem}[section]

\newtheorem{corollary}[theorem]{Corollary}
\newtheorem{lemma}[theorem]{Lemma}
\newtheorem{proposition}[theorem]{Proposition}
\newtheorem{fact}[theorem]{Fact}

\newtheorem*{theorem*}{Theorem}
\newtheorem*{AKpp}{Ax-Kochen Principle}
\newtheorem*{Artin}{Artin's Conjecture}

\theoremstyle{definition}
\newtheorem{definition}[theorem]{Definition}

\newtheorem{example}[theorem]{Example}
\newtheorem*{definition*}{Definition}

\theoremstyle{remark}
\newtheorem{remark}[theorem]{Remark}
\newtheorem{exercise}{Exercise}
\newtheorem{claim}{Claim}

\newcommand\R{\mathbb{R}}	
\newcommand\C{\mathbb{C}}
\newcommand\Q{\mathbb{Q}}
\newcommand\N{\mathbb{N}}
\newcommand\Z{\mathbb{Z}}
\newcommand\F{\mathbb{F}}

\newcommand\MM{\mathbb{M}}

\def\cM{\mathfrak{m}}

\def\cO{\mathcal{O}}

\def\cU{\mathcal{U}}

\def\Th{\operatorname{Th}}

\newcommand\LL{\mathscr{L}}

\newcommand\LLr{\mathscr{L}_{\mathrm{ring}}}

\newcommand{\ACF}{\mathrm{ACF}}

\newcommand{\Id}{\mathrm{Id}}

\newcommand{\Frac}{\mathrm{Frac}}

\newcommand{\alg}{\mathrm{alg}}
\newcommand{\Gal}{\mathrm{Gal}}

\newcommand{\ac}{\mathrm{ac}}

\newcommand{\tp}{\mathrm{tp}}

\newcommand{\DP}{\mathrm{dp}}
\newcommand\LLvf{\mathscr{L}_{\mathrm{vf}}}
\newcommand\LLres{\mathscr{L}_{\mathrm{res}}}
\newcommand\LLgp{\mathscr{L}_{\mathrm{gp}}}
\newcommand\LLdp{\mathscr{L}_{\mathrm{dp}}}

\newcommand\LLos{\mathscr{L}_{\mathrm{1s}}}
\newcommand\LLts{\mathscr{L}_{\mathrm{3s}}}

\newcommand{\rsa}{\rightsquigarrow}

\newcommand\Tres{T_{\mathrm{res}}}
\newcommand\Tgp{T_{\mathrm{gp}}}
\newcommand\TDP{T^{\mathrm{dp}}}

\def\seq{\subseteq}
\def\sneq{\subsetneq}
\newcommand{\set}[1]{\left\{ {#1} \right\}}
\newcommand{\vect}[1]{\langle {#1} \rangle}
\newcommand{\abs}[1]{\lvert {#1} \rvert}

\DeclareMathOperator{\fix}{fix}

\DeclareMathOperator{\Aut}{Aut}

\DeclareMathOperator{\res}{res}

\definecolor{airforceblue}{rgb}{0.36, 0.54, 0.66}

\def\IND{\setbox0=\hbox{$x$}\kern\wd0\hbox to 0pt{\hss$\mid$\hss}
\lower.9\ht0\hbox to 0pt{\hss$\smile$\hss}\kern\wd0}
\def\NotIND{\setbox0=\hbox{$x$}\kern\wd0\hbox to 0pt{\mathchardef\nn=12854\hss$\nn$\kern1.4\wd0\hss}\hbox to 0pt{\hss$\mid$\hss}\lower.9\ht0 \hbox to 0pt{\hss$\smile$\hss}\kern\wd0}

\makeatletter
\newcommand{\setword}[2]{%
  \phantomsection
  #1\def\@currentlabel{\unexpanded{#1}}\label{#2}%
}
\makeatother

\renewcommand{\models}{\vDash}

\begin{document}

\title{The Ax-Kochen-Ershov Theorem}

\author[C. d'Elb\'{e}e]{Christian d\textquoteright Elb\'ee}
\address{Mathematisches Institut der Universität Bonn\\
Office 4.004, Endenicher Allee 60\\ 53115 Bonn\\ Germany}
\urladdr{\href{http://choum.net/\textasciitilde chris/page\textunderscore perso/}{http://choum.net/\textasciitilde chris/page\textunderscore perso/}}

\date{\today}

\maketitle

\begin{abstract}
    
These are the notes of a course for the summer school \textit{Model Theory in Bilbao} hosted by the Basque Center for Applied Mathematics (BCAM) and the Universidad del País Vasco/Euskal Herriko Unibertsitatea in September 2023.

The goal of this course is to prove the Ax-Kochen-Ershov (AKE) theorem, see Theorem \ref{thm:AKE} below. This classical result in model theory was proven by Ax and Kochen and independently by Ershov in 1965-1966. The AKE theorem is considered as the starting point of the model theory of valued fields and witnessed numerous refinements and extensions. To a certain measure, motivic integration can be considered as such. The AKE theorem is not only an important result in model theory, it yields a striking application to $p$-adic arithmetics. Artin conjectured that all $p$-adic fields are $C_2$ (every homogeneous polynomial of degree $d$ and in $>d^2$ variable has a non trivial zero, see Definition \ref{def:C2}). A consequence of the AKE theorem is that the $p$-adics are \textit{asymptotically} $C_2$, in a sense that will be precised in Subsection \ref{sub:AKPARTIN}. The conjecture of Artin has been disproved by Terjanian in 1966, yielding that the solution given by the AKE theorem is in a sense optimal. The proof presented here is due to Pas but the general strategy stays faithful to the original paper of Ax and Kochen, which consist in the study of the asymptotic first-order theory of the $p$-adics.
\end{abstract}

\vspace{40pt}

\setcounter{tocdepth}{1}
\hrulefill
\tableofcontents	
\hrulefill

\section*{Introduction and preliminaries}

The object of study here are \textit{valued fields}, i.e. fields equipped with a \textit{valuation}.

\begin{definition*}
    A \textit{valuation} on a field $K$ is a group homomorphism $v : K^\times\to \Gamma$ where $(\Gamma,+,0,<)$ is an ordered abelian group which further satisfies \[v(a+b) \geq \min\set{v(a),v(b)}.\]
    We often extend $v$ to the whole $K$ by setting $v(0)>\gamma$ for all $\gamma\in \Gamma$, which is abbreviated by $v(0) = \infty$.
\end{definition*}

A typical example of a valued field is $\Q_p$, naturally equipped with the $p$-adic valuation $v_p$. A valuation on a field encompasses a rich set of data that we recall now. Let $(K,v)$ be a valued field.

\subsubsection*{Value group}
The \textit{value group} of $(K,v)$ is the subgroup $v(K^\times)$ of $\Gamma$. We often (but not always) assume that $\Gamma$ is the value group, i.e. that $v$ is onto. Note that as ordered abelian groups, $\Gamma$ and $v(K^\times)$ are torsion-free.

\subsubsection*{Valuation ring, maximal ideal} The set $\cO = \set{a\in K\mid v(a)\geq 0}$ is an integral domain which satisfies a very strong property: $a$ divides $b$ or $b$ divides $a$ (in $\cO$) for all $a,b\in \cO$. Domains satisfying this property are called \textit{valuation rings}. The spectrum of ideals is lineary ordered by inclusion and in particular, valuation rings are  \textit{local ring} i.e. they have a unique maximal ideal. The maximal ideal of $\cO$ is denoted $\cM$ and satisfies $\cM = \set{a\in K\mid v(a)>0}$. We refer to $\cO$ as \textit{the} valuation ring of $(K,v)$ and $\cM$ \textit{the} maximal ideal of $(K,v)$. The (multiplicative) group of units in $\cO$ is denoted $\cO^\times$. It is easy to check that $\cO^\times = \set{a\in K\mid v(a) = 0}$ and that the value group of $v$ is isomorphic to the quotient $K^\times/\cO^\times$.

\subsubsection*{Residue field} The quotient $\cO/\cM$ is a field called the \textit{residue field}, denoted $k$. The quotient map $\res:\cO\to k$ is called the \textit{residue map} and will play an essential role all along the paper.\\

We often recall the previous data via the following diagram.
\begin{center} 
\begin{tikzcd}
K \ar[r,"v"] \ar[d,"\res"] & \Gamma\cup\set{\infty}\\
k & 
\end{tikzcd}\\
\end{center}

A recurrent idea is that the valued field $(K,v)$ is ``controlled" by the value group and the residue field. It turns out that model theory is a nice setting to make this intuition concrete, as we will see with the AKE theorem.

The model-theoretic treatment of valued fields uses various languages, which are equivalent in the sense that they have the same first-order expressibility. Let $\LLr = \set{+,-,\cdot,0,1}$ be the language of rings, we generally use this language for rings and for fields.

\subsubsection*{Three-sorted language} The most intuitive way of encompassing the full structure of a valued field in a first-order language is by considering three sorts. Let $\LLts$ be the three sorted language defined by:
\begin{itemize}
    \item one sort for the valued field $K$ in a copy of the language of rings $\LLvf = \set{+,-,\cdot,0,1}$ (the \textit{valued field sort})
    \item one sort for the residue field $k$ in a different copy $\LLres = \set{+,-,\cdot,0,1}$ of the language of fields (the \textit{residue field sort})
    \item one sort for the value group $\Gamma$ in the language of ordered groups expanded by a constant $\LLgp = \set{+,-<,0,\infty}$ (the \textit{value group sort})
    \item a function symbol $v:K\to \Gamma\cup \set{\infty}$ for the valuation
    \item a function symbol $\res: K\to k$ for an extension of the residue map $\cO\to k$ to $K$.
\end{itemize}
Each valued field $(K,v)$ can be considered as a $\LLts$-structure by interpreting the right objects in the right sorts. The residue map $\res:\cO\to k$ will be extended to $K$ by setting $\res(K\setminus \cO) = \set{0}$. Note that the value group sort is a little more than a group because of $\infty$, and we extend the group structure so that $\gamma+\infty = \infty$, $-\infty = \infty$. As a multi-sorted structure, variables used to construct sentences and formulas are tagged by the sort they talk about. To make this apparent in $\LLts$, we will use $x,y,z,\ldots$ as variables for the valued field sort; $\xi,\zeta,\ldots$ for variables in the value group sort; and $\bar x, \bar y, \bar z,\ldots$ for the residue field sort. It is easy to write down an $\LLts$-theory whose models are exactly valued fields in which $v,\res$ are onto and $\res\upharpoonright \cO $ is the residue map\footnote{Once $v$ has been specified to be a valuation, one just has to say that if $v(x)\geq 0, v(y)\geq 0$ then $\res(x) = \res(y)\iff v(x-y)>0$ and if $v(x)<0$ then $\res(x) = 0$.} and every given valued field $(K,v)$ can be seen as an $\LLts$-structure, usually denoted $(K,k,\Gamma)$ or $(K,k,\Gamma,v,\res)$.

\subsubsection*{One-sorted language} Let $\LLos =\LLr\cup\set{P}$ where $P$ is a unary predicate. A valued field $(K,v)$ is often considered in the more economical language $\LLos$ by letting $P$ be a predicate for the valuation ring $\cO$. From a valued field $(K,\cO)$ in $\LLos$, we can recover the three-sorted structure $(K,k,\Gamma)$
. The valuation function is interpretable as the canonical projection $K^\times\to K^\times/\cO^\times$ from $K^\times $ to the sort $K^\times/\cO^\times$. For instance, 
\begin{align*}
    v(a) = v(b)&\iff v(ab^{-1}) = 0\\
    &\iff ab^{-1}\in \cO^\times\\
    &\iff \exists y\ y\in \cO\wedge ya = b.
\end{align*}

Statements about elements of $K^\times/\cO^\times $ reduce to statements in $(K,\cO)$. 
The ordered group structure on the imaginary sort $K^\times/\cO^\times$ is also definable: for instance, one defines the order on $K^\times/\cO^\times$ by $v(a)\leq v(b)\iff ba^{-1}\in \cO$. The residue map and the residue field are also interpretable as the canonical projection $\res:\cO\to k = \cO/\cM$, extending to $K\setminus \cO$ by $0$. One sees that the interpretation is uniform, in the sense that it follows the same procedure from any $\LLos$ valued field $(K,\cO)$ using only that $\cO$ is a valuation ring.

\subsubsection*{Ring language for the valuation ring} The most economic way of treating a valued field $(K,v)$ in first-order logic is by considering the valuation ring $\cO$ in the language $\LLr$. From $\cO$ we recove $K$ which is the fraction field of $\cO$ (which is interpretable as the quotient of $\cO\times \cO$ by the definable relation $(x,y)\sim (z,t)\iff xt-yz=0$) as well as a copy $\cO'$ of $\cO$ in $K$ (which consists of the image of elements of the form $(a,1)$ in the projection $\pi:\cO\times \cO\to K$).

We see here that the model-theory of valued field essentially reduces to the model-theory of valuation rings, but difference between languages might still be relevant, especially in Section \ref{sec:pastheorem} where we will need to be more explicit about the value group and the residue field. We will allow ourselves to freely switch from one language to another when considering a given valued field although most of the time the three-sorted language $\LLts$ will be preferred.\\

\begin{center} 
\begin{tikzcd}
 & (K,v) \arrow[dl, rightsquigarrow] \arrow[d, rightsquigarrow] \arrow[dr, rightsquigarrow] &\\
\cO \ar[r, leftrightarrow] & (K,\cO) \ar[r, leftrightarrow] & (K,k,\Gamma)\\
\LLr \ar[u,dash] & \LLos \ar[u,dash] & \LLts \ar[u,dash]
\end{tikzcd}
\end{center}

\begin{exercise}
    Write down (or convince yourself that it exists) the following.
    \begin{enumerate}
        \item The $\LLts$-theory $T_\mathrm{ts}$ of valued field (with surjective valuation $v$).
        \item The $\LLos$-theory $T_\mathrm{os}$ of valued fields.
        \item The $\LLr$-theory $T_\mathrm{vr}$ of valuation rings.
    \end{enumerate}
\end{exercise}

\subsection*{Notations and conventions} In a valued field $(K,v)$, for consistency with the notations of the associated three-sorted structure $(K,k,\Gamma)$, we will generally use the variables $a,b,c,\ldots$ for elements of the field $K$, $\alpha,\beta,\gamma,\ldots$ for elements of the value group $\Gamma$ and $\bar a,\bar b,\ldots$ for elements of the residue field $k$.

\clearpage

\section{The Ax-Kochen-Ershov Theorem}

\subsection{$p$-adic numbers}

Let $p$ be a prime number. The ring $\Z_p$ of $p$-adic integers is for us the set of formal sums:
\[\sum_{i\in \N} a_i p^i \quad a_i\in \set{0,\ldots,p-1}.\]
There is a unique way to represent elements\footnote{Note that if every finite sum $\sum_{i=0}^n a_ip^i$ represents an element of $\N$, elements of $\Z$ can be infinite sums: $-1 = \sum_{i\in \N} (p-1)p^i$.} of $\Z$ (even of $\Z_{(p)}$) in $\Z_p$ and addition and multiplication in $\Z_p$ are the ones extending addition and multiplication in $\Z$ in the most natural way (i.e. addition componentwise with reminder, and distributive multiplication with reminder).

The ring of $p$-adic integers is usually defined as the inverse limit of the family of rings $(\Z/p^n\Z)_n$ (with natural epimorphism $\Z/p^m\to \Z/p^n\Z$ for $m\geq n$) hence $p$-adic integers as sequences $(a_i)_{i\in \N}$ such that $a_i\equiv a_{j}\ (\mathrm{mod}\  p^i)$ for $i\leq j$. This gives a more formal construction but is equivalent in the end\footnote{Note that the correspondence is \textit{not} via $(a_i) \leftrightarrow \sum_i a_i p^i$ (but is it close).}.

The ring $\Z_p$ is a local domain with maximal ideal $p\Z_p$. Even more, it is a valuation ring and its field of fraction is denoted $\Q_p$, the field of $p$-adic numbers. The representation as infinite sum of $\Z_p$ extends to $\Q_p$ by letting the index rang over integer: every element in $\Q_p$ is a sum
\[\sum_{i\geq i_0} a_ip^i \quad a_i\in \set{0,\ldots,p-1}\]
for some $i_0\in \Z$. The map $v_p: \Q_p \to \Z\cup\set{\infty}$ defined by $v(x) = \infty\iff x = 0$ and for $\sum_{i} a_ip^i\neq 0$
\[
    v(\sum_{i} a_ip^i) = \min\set{i\mid a_i \neq 0} 
\]
defines a valuation on $\Q_p$, called the \textit{$p$-adic valuation}. The valuation ring associated is $\cO = \Z_p$, the maximal ideal is $\cM = p\Z_p$, the value group is $\Gamma = \Z$ and the residue field is $\Z_p/p\Z_p\cong \Z/p\Z = \F_p$. We summarize this data by the following:
\begin{center} 
\begin{tikzcd}
\Q_p \ar[r,"v_p"] \ar[d,"\res"] & \Z\cup\set{\infty}\\
\F_p & 
\end{tikzcd}\\
\end{center}

As Kurt Gödel used to say\footnote{Quoted from Christopher Nolan's \underline{Oppenheimer} movie.}: \textit{Trees are the most inspiring structures.} For a model-theorist point of view, the structure that encompasses the combinatorial aspect of valued fields (and especially $\Q_p$) is the one of a tree. We describe how this is done, by representing $\Z_p$ as a tree. One thinks of elements of $\Z_p$ as \textit{branches} of an infinite tree rooted in one single point with $p$ branches at each note, representing the choice of the coefficient $a_i$ of the term $a_i p^i$. Hence the nodes on each branches are indexed by the positive part of the value group (here $\N$) and the choices at each nodes represents the residue field $\F_p$.
\[\sum_i a_i p^i  =\text{ the branch choosing $a_i$ at the $p^i$-th level}\]
We refer to Figure \ref{fig:padic} for a sketch of the tree representation of $\Z_2$. In the representation of $\Z_p$ as a tree, there is a special branch: the $0$ branch. We represent it on the far left of the drawing and it is the branch choosing $0$ at each level. Every element branching on the zero branch at a level say $p^{i_0}$ is written $a = \sum_{i\geq i_0} a_i p^{i}$ so $v(a) = i_0$. More generally the node where two elements $a$ and $b$ branch is at the level corresponding to the valuation of $a-b$: below this branching, $a$ and $b$ made the same choices of $a_i$'s hence the difference cancel this prefix. 
\tikzset{every picture/.style={line width=0.75pt}} 

\begin{figure}
    \centering
    \includegraphics[ scale=.4]{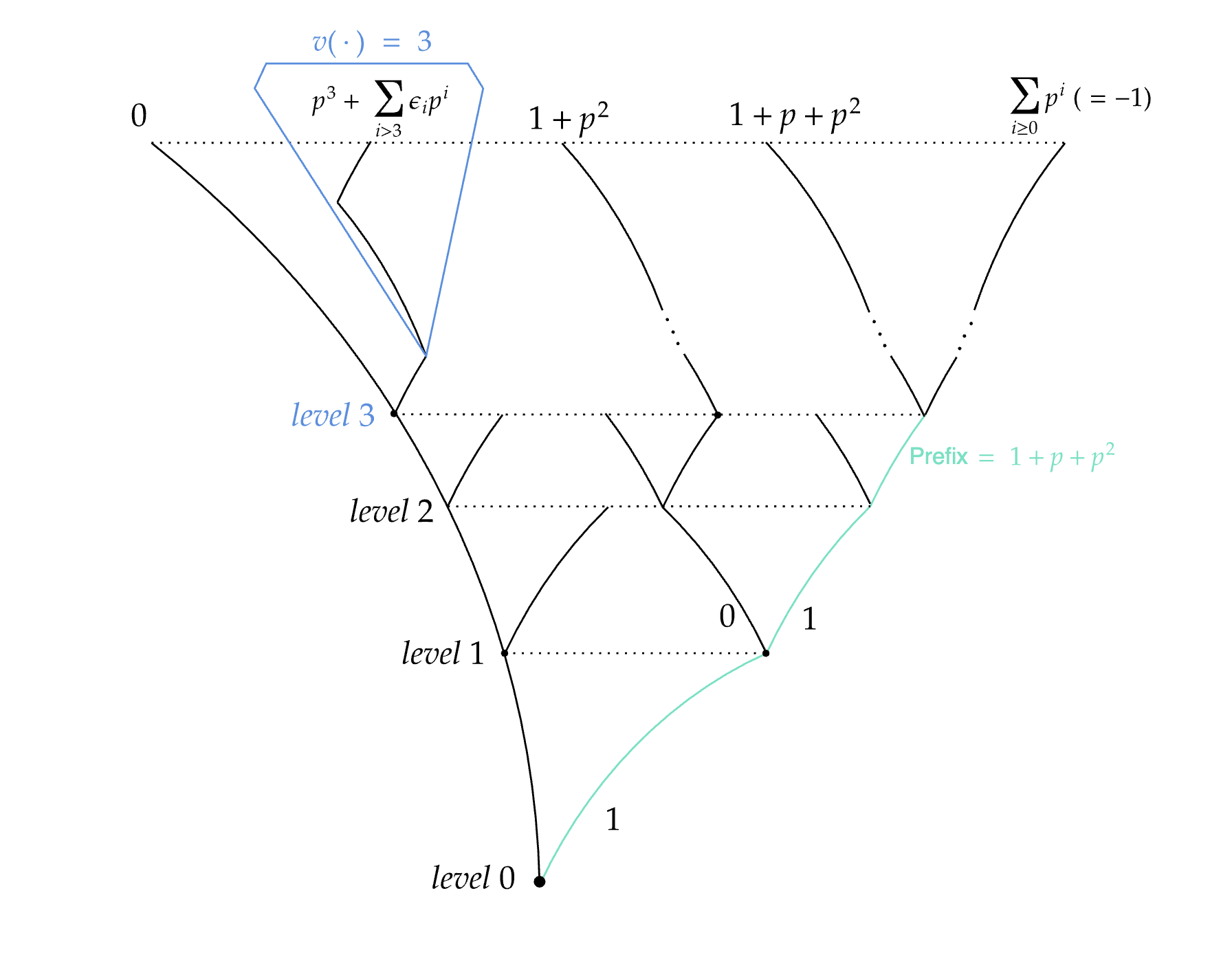}
    \caption{Tree representation of $\Z_p$, for $p=2$}
    \label{fig:padic}
\end{figure}


\subsection{Formal Laurent series} For any field $K$ the field of formal Laurent series $K((t))$ is the set of formal sums $\sum_{i\geq i_0} a_i t^i$ for $a_i\in K$ and $i_0\in \Z$ with obvious addition and multiplication naturally extending the ones of polynomials in $t$. The field $K((t))$ can be equipped with the \textit{$t$-adic valuation} $v_t$ defined as follows
\[v_t(\sum_{i} a_i t^i) = \min\set{i\in \Z\mid a_i\neq 0}.\]
The associated valuation ring is the ring $K[[t]]$ of formal series $\sum_{i\geq 0}a_i t^i$, the residue field is $K$ and the value group is $\Z$.
\begin{center} 
\begin{tikzcd}
K((t)) \ar[r,"v_t"] \ar[d,"\res"] & \Z\cup\set{\infty}\\
K & 
\end{tikzcd}\\
\end{center}

For a prime number $p$, the tree representation of $\F_p[[t]]$ can be done exactly as for the $p$-adic integers, by representing elements $\sum_{i\geq 0} a_i t^i$ as branches and the path represents the choices of the sequence $(a_i)_{i\geq 0}$, one gets the same tree as for $\Z_p$.

Those tree representations do not reflect the arithmetic in $\Z_p$ or $\F_p[[t]]$ but it enlightens a strong similarity between $\Z_p$ and $\F_p[[t]]$, which is precisely what the Ax-Kochen principle is all about.

\subsection{The Ax-Kochen principle and Artin's conjecture}\label{sub:AKPARTIN}

Recall that in model theory, rings and fields are often considered in the language of (unital) rings $\LLr = \set{+,-,\cdot,0,1}$. The goal of this course is to present the following transfer theorem proved in 1965 by Ax and Kochen \cite{AK65I}.

\begin{AKpp}\label{AKpp}
    Let $\theta$ be any sentence of the language $\LLr$. Then for all but finitely many prime numbers $p$, we have
    \[\Z_p\models \theta\iff \F_p[[t]]\models \theta .\]
    Equivalently, for any sentence $\theta$ in the three-sorted language $\LLts$, then for all but finitely many $p$ we have
    \[(\Q_p,\F_p,\Z)\models \theta \iff (\F_p((t)), \F_p,\Z )\models \theta.\]
\end{AKpp}

The fact that the two statements are equivalent comes from the fact that the ring $\Z_p$ is bi-interpretable with the three-sorted structure $(\Q_p,\F_p,\Z)$ and similarly for $\F_p[[t]]$ and $(\F_p((t)),\F_p,\Z)$.

We will see that this theorem follows from an important quantifier elimination result, the Ax-Kochen-Ershov theorem (see Theorem \ref{thm:AKE} below). The idea behind the Ax-Kochen principle is that $(\Q_p)_p$ and $(\F_p((t)))_p$ ``asymptotically" share the same first-order theory. Before going into those considerations, we state an important application of the Ax-Kochen principle on a conjecture of Artin.

\noindent\textbf{A little history.} It all starts in 1933 when Tsen proves that if $K$ is an algebraically closed field, then there are no nontrivial central division algebras over the field $K(X)$. 
Reading on Tsen's work, Artin isolated the property that $K(X)$ satisfies and which prevents central division algebras over $K(X)$ to exist. This property --called \textit{quasi-algebraically closed} at that time-- corresponds to the property $C_1$: for all $d\in \N^{>0}$, an homogeneous polynomial in $>d$ variables has a nontrivial zero. The notion were later generalised by Lang.

\begin{definition}[Lang]\label{def:C2}
    For $d,i\in \N^{>0}$, we say that a field $K$ is $C_i(d)$ if every homogeneous polynomial of degree $d$ with $>d^i$ variables with coefficients in $K$ has a nontrivial zero in $K$. A field is $C_i$ if it is $C_i(d)$ for all $d\in \N^{>0}$.
\end{definition}

\begin{remark}\label{rk:C2onlyind2+1}
    A field $K$ is $C_i(d)$ if and only if every homogeneous polynomial of degree $d$ with $d^i+1$ variables with coefficients in $K$ has a nontrivial zero in $K$, for $d>1$. See Exercise \ref{exo:homogeneouspolynomials}.
\end{remark}

Existence of central division algebras over a given field are intrinsically linked to solutions of certain polynomial equations. A famous theorem of Wedderburn yields that there are no central division algebra over a finite field, hence Artin (and already Dickson before him) conjectured that every finite field is $C_1$. This was proved by Chevalley in 1935 \cite{Chev35}.

\begin{fact}[Chevalley\footnote{This theorem were later extended by Warning (Chevalley-Warning Theorem) and then by Ax \cite{ax64} to the following stronger form. Let $K$ be a finite field with $q$ elements of characteristic $p$. Let $f$ be a polynomial of degree $d$ in $n$ variables with coefficient in $K$. Let $N(f)$ be the number of distinct zero of $f$ in $K^n$. If $n>d$ and $a$ is the largest integer strictly less than $\frac{n}{d}$ then $q^a$ divides $N(f)$. In particular, if $f$ has no constant term, then $\vec 0$ is a zero of $f$ hence there exists a nontrivial zero of $f$.}, 1935]
    Every finite field is $C_1$.
\end{fact}

Later, the result of Tsen were to be generalised in various forms, for instance if $K$ is $C_i$ then $K(X_1,\ldots,X_j)$ is $C_{i+j}$ (Greenberg \cite{Greenbegr67}) and if $K$ is $C_i$ then $K((t))$ is $C_{i+1}$ (Greenberg, \cite{Greenberg66}). Together with Chevalley's result, we obtain a result already proved by Lang in 1952 \cite{Lang52}.

\begin{fact}[Lang, 1952]
    $\F_p((t))$ is $C_2$, for all $p$.
\end{fact}

Concerning the $p$-adics, a hundred years ago, H. Hasse \cite{Hasse23} proved that  every quadratic form (i.e. homogeneous polynomial of degree $2$) over $\Q_p$ in $5$ variables have a nontrivial zero in $\Q_p$. In other words, $\Q_p$ is $C_2(2)$. The existence of normic forms of order $2$ (i.e. forms of degree $d$ in $d^2$ variables without nontrivial zeros) on $\Q_p$ prevent $\Q_p$ to be $C_1$. In 1936, Artin made the following conjecture.

\begin{Artin}[1936]
$\Q_p$ is $C_2$, for all $p$.
\end{Artin}

In 1952, Lewis \cite{lewis52} proved that $\Q_p$ is $C_2(3)$, a new step toward the proof of the conjecture. In 1965, Ax and Kochen used Lang's result to get the following ``asymptotic" solution to Artin's conjecture:

\begin{corollary}\label{cor:artinssolution}
    For all $d\in \N$, there exists $N = N(d)$ such that $\Q_p$ is $C_2(d)$ for all $p>N$. In other words, for each $d\in \N$, the set of $p$ such that $\Q_p$ is not $C_2(d)$ is finite.
\end{corollary}

They used the Ax-Kochen principle as follows.

\begin{proof}
    Let $d\in \N$ and $m = d^2+1$. Consider the list $(M_i(X_1,\ldots,X_m))_{1\leq i \leq l}$ of all monomials of degree $d$ and for each $\vec x = (x_1,\ldots,x_l)$ introduce the notation \[P_{\vec x}(\vec X) := \sum_{i = 1}^l x_iM_i(\vec X)\]
    For any field $K$, the set $\set{P_{\vec a}(\vec X)\mid \vec a\in K^l}$ consists of all homogeneous polynomials of degree $d$ in $\leq m$ variables.

    Let $\theta_d$ be the following sentence:
    \[\forall x_1,\ldots, x_l \left[ \underbrace{\vec x \neq \vec 0}_{\text{$P_{\vec x}(\vec X)$ is not the zero polynomial}} \rightarrow \underbrace{\exists z_1,\ldots, z_m  \left(  P_{\vec x}(\vec z) = 0  \wedge \vec z\neq \vec 0 \right)}_{\text{$P_{\vec x}(\vec X)$ has a nontrivial zero}}  \right] .\] 
    By Remark \ref{rk:C2onlyind2+1}, for any field $K$, we have $K\models \theta_d$ if and only if $K$ is $C_2(d)$. Note that if $P_{\vec x}(\vec X)$ uses strictly less than the variables in $\vec X$, then it is trivial that is has a nontrivial zero.
    By the Ax-Kochen principle, there exists $N = N(\theta) = N(d)$ such that for all $p> N$ we have $\F_p((t))\models \theta_d$ if and only if $\Q_p\models \theta_d$. By Lang's theorem $\F_p((t))$ is $C_2$ for all $p$ hence for $p>N$ we have $\Q_p\models \theta_d$. 
\end{proof}

This solution is only asymptotic and does not fully answer the question asked by Artin. Considering that Artin's conjecture is false in general, this asymptotic solution is not so bad afterall. Indeed, around the same time as Ax and Kochen's solution, Guy Terjanian found the first counterexample to Artin's conjecture.

\begin{example}[$\Q_2$ is not $C_2(4)$]
    Consider $F(x_1,\ldots,x_{18})$ to be the form:
    \[G(x_1,x_2,x_3)+G(x_4,x_5,x_6)+G(x_7,x_8,x_9) + 4G(x_{10},x_{11},x_{12})+ 4G(x_{13},x_{14},x_{15})+ 4G(x_{16},x_{17},x_{18})\]
    for $G(x,y,z) = x^4+y^4+z^4-(x^2y^2+x^2z^2+y^2z^2) - xyz(x+y+z)$. Then Terjanian proves in \cite{Terjanian66} that the only zero of $F$ in $\Q_2$ is the trivial one, so $\Q_2$ is not $C_2(4)$.
\end{example} 

More example were found afterwards however in each case with $d$ even. It is a current open question whether Artin's conjecture is true for odd $d$, in other words, is every $\Q_p$ $C_2(d)$ for all odd $d$? In particular it is still open whether every $\Q_p$ is $C_2(5)$. See \cite{HeathbrownICM} for more on that topic.

\begin{remark}[On bounds]
    There exists an explicit bound for the value of $N(d)$ in Corollary \ref{cor:artinssolution}. In \cite{Brown78}, Brown proved that $N(d)$ can be choosen to be
\[2^{2^{2^{2^{2^{11^{d^{4d}}}}}}}.\]

The same method that we will use to prove the Ax-Kochen principle also yields that the first-order theory of the field $\Q_p$ is decidable. This appear first in \cite{AK65II}. Given $(p,d)$, there exists a procedure for deciding whether the statement $\theta_d$ (from the proof of Corollary \ref{cor:artinssolution}) is true or false in $\Q_p$. Hence for a fixed degree $d$, one could in theory use an algorithm to check that $\Q_p$ is $C_2(d)$ for all prime $p$ lower than Brown's bound. In the particular case of $d = 5$, Heath-Brown \cite{HeathbrownICM} proved that the bound can be reduced to $17$, but as he puts it ``This is certainly decidable in principle, but whether it is realistic to expect
a computational answer with current technology is unclear."
\end{remark}

\begin{remark}
    Note that being $C_i$ is equivalent to the following stronger formulation: $K$ is $C_i$ if for all $f_1,\ldots,f_r$ homogeneous polynomials in $n$ variables of degree $d$ with $n>rd^i$ there exists a nontrivial common zero of $f_1,\ldots,f_r$. This result is attributed to Lang and Nagata, see \cite{Greenbegr67}.
\end{remark}

\begin{exercise}
    Prove that $\R$ is not $C_i(2d)$ for any $i,d\in \N^{>0}$.
\end{exercise}

\begin{exercise}\label{exo:homogeneouspolynomials}
    Consider a homogeneous polynomial $f\in K[X_1,\ldots, X_{n+1}]$ of degree $d$ in $X_1,\ldots,X_{n+1}$ variables. Assume that $f$ has no nontrivial zeros in $K$, then $g(X_1,\ldots,X_{n}) = f(X_1,\ldots,X_n,0)$ is homogeneous of the same degree as $f$ and has no nontrivial zeros in $K$.
    \begin{enumerate}
        \item Prove that $X_{n+1}$ does not divide $f$.
        \item Deduce that $g(X_1,\ldots,X_n)$ is nonzero. 
        \item Conclude.
    \end{enumerate}
\end{exercise}

\subsection{Henselian valued fields}

The $p$-adics and other valued fields that we will consider here share a very important property which we define now.

\begin{definition}
    A valued field $(K,v)$ is \textit{Henselian} if it satisfies the following property:
    
    \textbf{Simple zero lift}. For each $P\in \cO[X]$ and $\bar a\in k$ such that $\res(P)(\bar a) = 0$ and $\res(P')(\bar a) \neq  0$ there exists $b\in \cO$ such that $P(b) = 0$ and $\res(b) = \bar a$.
\end{definition}

\begin{remark}\label{rk:henselianisfirstorder}
In whatever language considered to study a valued field $(K,v)$, being Henselian is a first-order property. Indeed, $\cO$, $k$ and the map $\res:\cO\to k$ are interpretable. Write $P_{\vec x}(y) = \sum_{i=0}^n x_iy^i$ and $P'_{\vec x}(y) = \sum_{i=1}^n i x_i y^{i-1}$ and the set of sentences
\[\forall x_0\ldots x_n \exists \bar y \left[ (P_{\res(\vec x)} (\bar y) = 0\wedge P'_{\res(\vec x)}(\bar y)\neq 0)\to \exists y P_{\vec x}(y) = 0\wedge  \res(y) = \bar y\right] \]
for all $n\in \N^{>0}$ is satisfied by a valued field if and only if it is Henselian.
\end{remark}

\begin{fact}
    $(\Q_p,v_p)$ and $(K((t)),v_t)$ are Henselian valued field.
\end{fact}

\begin{remark}
    In $\Q_p$ we have the $p$-adic absolute value given by $\abs{a}_p = p^{-v_p(a)}$. We have $\abs{a+b}_p\leq \max\set{\abs{a}_p,\abs{b}_p}$ and $\abs{\cdot}_p$ endows $\Q_p$ with an ultrametric for which $\Q_p$ is complete. Using Newton's method, one gets a results due to Hensel, that every complete valued field with $v(K^\times)\seq \R$ (archimedean value group) satisfies the simple zero lift property. 
\end{remark}

\begin{remark}\label{rk:henselslemma}
    There seems to be an ambiguity in the litterature about what Hensel's lemma really is. For a few authors, Hensel's lemma is the fact that complete valued fields with archimedean value group satisfy the simple zero lift, or an equivalent statement such as: 
    \begin{center}
    $(\ast)$ given $a\in \cO$ and $P\in \cO[X]$ with $v(P(a))>2v(P'(a))$, there exists $b\in \cO$ with $P(b) = 0$ and $v(b-a)>v(P'(a)))$
    \end{center}
    For most authors, this statement $(\ast)$ (or any of its variant, or the simple zero lift property) is Hensel's lemma itself. See Exercise \ref{exo:henselslemma} for a proof that the simple zero lift is equivalent to $(\ast)$.
\end{remark}

\subsection{The Ax-Kochen-Ershov Theorem}

The Ax-Kochen principle follows from a theorem proved in 1965 by Ax and Kochen \cite{AK65I} and independently on the other side of the iron curtain by Ershov \cite{ershov65}.

Given any valued field $(K,v)$ with value group $\Gamma$ and residue field $k$, there are three cases for the pair of characteristics $(\mathrm{char}(K), \mathrm{char}(k))$:
\begin{itemize}
    \item (Equicharacteristic $0$) $(0,0)$ this happens if $\mathrm{char}(k) = 0$. Examples are $\C((X))$, $\R((X))$. 
    \item (Equicharacteristic $p$) $(p,p)$ this happens if $\mathrm{char}(K) = p$. Examples are $\F_p((X))$, $\F_p^\alg((X))$.
    \item (Mixed characteristic) $(0,p)$ this happens if $v(p)>0$. An example is $\Q_p$.
\end{itemize}

\begin{remark}
    As $v(1) = 0$, in equicharacteristic $0$, one has $v(n) = 0$ for all $n\in \Z$, hence as elements of valuation $0$ are invertible in the valuation ring, one has $\Q^\times\seq \cO^\times$.
\end{remark}

Recall that for two structures $M,N$ in the same language $\LL$, we write $M\equiv N$ (in $\LL$) if every $\LL$-sentence true in $M$ is also true in $N$ and vice versa. Recall that a valued field $(K,v)$ can be consider in various equivalent languages, in particular it can be seen as an $\LLts$-structure $(K,k,\Gamma)$ where $k$ is the residue field and $\Gamma$ the valued group.

\begin{theorem}(Ax-Kochen-Ershov)\label{thm:AKE}
    Let $(K,k_K,\Gamma_K)$ and $(L,k_L,\Gamma_L)$ be two valued fields in the three-sorted language $\LLts$ which are Henselian and of equicharacteristic $0$. Then
    \[(K,k_K,\Gamma_K)\equiv (L,k_L,\Gamma_L) \quad \text{ as valued fields in $\LLts$}\iff \begin{cases}
        k_K\equiv k_L &\text{ (as fields in $\LLres$) and }\\
        \Gamma_K\equiv \Gamma_L & \text{ (as ordered groups in $\LLgp$)}
    \end{cases}
    \]
\end{theorem}

This theorem will follow from Pas' theorem, which we will prove in the next section. 

Let us see now how the Ax-Kochen principle follows from this result.

\begin{proof}[Proof of the Ax-Kochen principle from Theorem \ref{thm:AKE}]
    By contradiction assume that $\theta$ is an $\LLts$ sentence such that for some infinite subset $S$ of prime numbers we have $(\Q_p,\F_p,\Z)\models \theta$ for all $p\in S$ and $(\F_p((t)),\F_p,\Z)\models \neg \theta$ for all $p\in S$.

    Let $\cU$ be a non-principal ultrafilter on the set of primes such that $S\in \cU$. Consider the $\LLts$-structures
    \[(K,k_K,\Gamma_K) = \prod_{\cU} (\Q_p,\F_p,\Z)\text{ and }(L,k_L,\Gamma_L) = \prod_{\cU} (\F_p((t)),\F_p,\Z).\]
    $K$ and $L$ are valued fields. Let $\sigma_{p}$ be the $\LLr$ sentence expressing 
    \[\underbrace{1+\ldots+1}_{\text{$p$ times}} = 0.\]
    For all $q$, $\Q_p\models \neg \sigma_q$ hence by \L o\'s theorem, $K$ is of characteristic $0$. Similarly, for all but one $q$ we have $\F_p((t))\models \sigma_q$ hence $L$ is of characteristic $0$. Both $k_K$ and $k_L$ are the pseudo-finite field $\prod_{\cU} \F_p$. For all but one $q$, we have $\F_p\models \neg \sigma_q$ hence $k_K$ and $k_L$ are of characteristic $0$. We conclude that both $(K,v)$ and $(L,v)$ are of equicharacteristic $0$. The value groups $\Gamma_K$ and $\Gamma_L$ equal $\prod_{\cU} \Z$ in both cases. By Remark \ref{rk:henselianisfirstorder} we have that both $K$ and $L$ are Henselian. By Theorem \ref{thm:AKE} we conclude that $(K,k_K,\Gamma_K)\equiv (L,k_L,\Gamma_L)$, however by \L o\'s theorem, $(K,k_K,\Gamma_K)\models \theta$ and $(L,k_L,\Gamma_L)\models \neg \theta$, a contradiction.
\end{proof}

\begin{remark}
    The proof shows that for all ultrafilter $\cU$ on the prime numbers, we have
    \[\prod_{\cU} (\Q_p,\F_p,\Z)\equiv \prod_{\cU} (\F_p((t)),\F_p,\Z).\]

\end{remark}

Here is a direct consequence of the AKE Theorem.
\begin{corollary}
For any fields $K,L$ of characteristic $0$ we have, as rings
   \[K\equiv L\iff K[[t]]\equiv L[[t]]\]
   In particular we have $\Q^\alg[[t]]\equiv \C[[t]]$. Note however that $\Q^\alg[t]\not \equiv \C[t]$ and $\Q^\alg[[t_1,t_2]]\not \equiv \C[[t_1,t_2]]$.
\end{corollary}

\subsection{Generalised series and further application}

Given a field $k$ and an ordered abelian group $\Gamma$, we now define a field $k((t^\Gamma))$ of \textit{generalised series} (or \textit{Hahn series}) with a valuation $v$ having residue field $k$ and value group $\Gamma$.

Recall that a subset $A$ of $\Gamma$ is \textit{well-ordered} if each nonempty subset of $A$ has a least element.

We define $K = k((t^\Gamma))$ to be the set of formal series 
\[f(t) = \sum_{\gamma\in \Gamma} a_\gamma t^\gamma\]
such that the support $\mathrm{supp}(f) = \set{\gamma\in \Gamma\mid a_\gamma\neq 0}$ is well-ordered. Using Exercise \ref{exo:lemmaHahnseries} the binary operations
\begin{align*}
    \sum a_{\gamma} t^\gamma + \sum b_\gamma t^\gamma &:= \sum (a_\gamma+b_\gamma)t^\gamma\\
    \left(\sum a_{\gamma} t^\gamma\right) \left(\sum b_\gamma t^\gamma\right) &:= \sum_\gamma\left(\sum_{\alpha+\beta = \gamma} a_\alpha b_\beta\right) t^\gamma
\end{align*}
are well-defined and turn $K$ into an integral domain. Further, we define a valuation on $K$:
\[v(\sum a_\gamma t^\gamma) = \min\set{\gamma\mid a_\gamma\neq 0}.
\]
\begin{theorem}
    For all $k,\Gamma$, the ring $K = k((t^\Gamma))$ is a field, $v$ is a valuation on $K$ and $(K,v)$ has residue field $k$ and value group $\Gamma$. The valuation ring is denoted $k[[t^\Gamma]]$, it consists of elements of $K$ of positive support.
\end{theorem}
\begin{proof}
    The proof of this result is mainly checking facts, and left as an exercise. The fact that $K$ is a field is a bit more involved and is detailed in Exercise \ref{exo:lemmaHahnseries2}.
\end{proof}

Some well-known facts about generalised power series, that we will not have time to prove in this course:
\begin{enumerate}
    \item If $k$ is algebraically closed and $\Gamma$ is divisible, then $k((t^\Gamma))$ is algebraically closed.
    \item $k((t^\Gamma))$ is Henselian, for all $k$ and $\Gamma$.
\end{enumerate}

This generalised series construction allows us to construct many new examples of Henselian valued field, by varying the residue field ($k = \R,\C,\F_p,\Q_p,...$) or the value group ($\Gamma = \R,\Q,\Z,\Z_p,\Z\times \Z,...$). As particular examples of generalised series, we recover the Laurent series $k((t))$, for $\Gamma = (\Z,+,<)$ and in particular our important example $\F_p((t))$ above.

\begin{corollary}\label{cor:hahnseriesAKE}
Here are some more consequences of the AKE Theorem \ref{thm:AKE} with the above facts on generalized series.
\begin{enumerate}
   \item For any henselian valued field $(K,v)$ of equicharacteristic $0$, residue field $k$ and value group $\Gamma$, we have
   \[K \equiv k((t^\Gamma))\]
   as valued fields.
   \item For any non-principal ultrafilter $\cU$ on the prime numbers, we have \[\prod_\cU \Q_p \equiv F((t))\]
   as valued fields, where $F = \prod_\cU \F_p$.
\end{enumerate}
\end{corollary}

Looking back at the tree representations of the valuation rings $\Z_p$ and $\F_p[[t]]$, one can also represent elements of $k[[t^\Gamma]]$ as branches of a tree. The tree itself might bit more abstract (e.g. if the value group is dense) but the tree representation still makes sense. In $\Z_p$ or $\F_p[[t]]$ the ``branching" of two elements $a$ and $b$ is at the level $v(a-b)$. One has to take into account that for an arbitrary $\Gamma$ and $a,b\in k[[t^\Gamma]]$, the ``branching" of $a$ and $b$ might not be an identified point. For instance assume that the support of $a$ and $b$ is $\gamma_{0}<\gamma_1<\ldots<\gamma_{\omega}$ ordered as the ordinal $\omega+1$\footnote{Take for instance $\Gamma = \Q$, $\gamma_i = \sum_{j=0}^i 2^{-j}$ for $i<\omega$ and $\gamma_{\omega} = 2$.} and assume that $a_{\gamma_i} = b_{\gamma_i}$ for $i<\omega$ and $a_\omega \neq b_\omega$. There is no branching point between $a$ and $b$ but the valuation of $a-b$ (namely $\gamma_\omega$) is very close to where the branching point should be. In effect, taking arbitrary elements $a$ and $b$, the valuation of $a-b$ is 
\[\sup \set{\gamma\in A\mid a_\alpha = b_\alpha \text{ for all $\alpha<\gamma$, $\alpha\in A$}}\]
where $A$ is the (well ordered) union of $\mathrm{supp}(a)$ and $\mathrm{supp}(b)$. In our representation of valued fields as trees, we will identify the (possibly imaginary) branching point of $a$ and $b$ with the valuation of $a-b$.


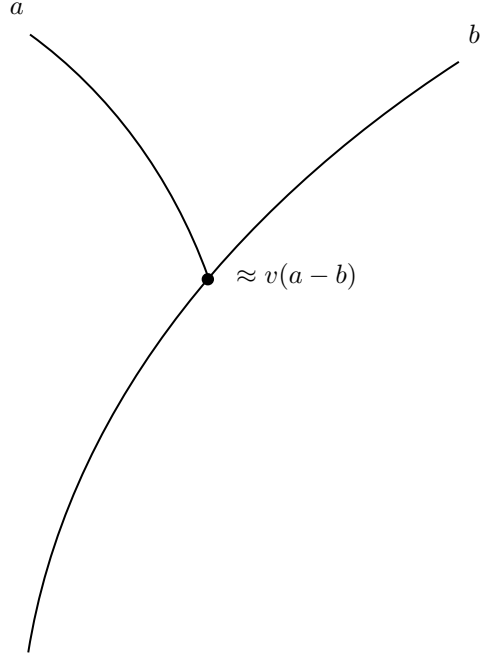
\begin{figure}
    \centering
    \begin{tikzpicture}[x=0.75pt,y=0.75pt,yscale=-1,xscale=1]

\draw  [draw opacity=0] (307.67,398.79) .. controls (326.15,282.62) and (405.04,178.69) .. (522.63,101.71) -- (1022,451.5) -- cycle ; \draw   (307.67,398.79) .. controls (326.15,282.62) and (405.04,178.69) .. (522.63,101.71) ;  
\draw  [draw opacity=0] (308.54,87.98) .. controls (348.02,116.86) and (379.26,159.26) .. (397.31,209.49) -- (185.75,309.5) -- cycle ; \draw   (308.54,87.98) .. controls (348.02,116.86) and (379.26,159.26) .. (397.31,209.49) ;  
\draw  [line width=2.25]  (395.73,211.08) .. controls (395.73,210.21) and (396.44,209.49) .. (397.31,209.49) .. controls (398.19,209.49) and (398.9,210.21) .. (398.9,211.08) .. controls (398.9,211.96) and (398.19,212.67) .. (397.31,212.67) .. controls (396.44,212.67) and (395.73,211.96) .. (395.73,211.08) -- cycle ;

\draw (297,70.4) node [anchor=north west][inner sep=0.75pt]    {$a$};
\draw (526,81.4) node [anchor=north west][inner sep=0.75pt]    {$b$};
\draw (410,201.4) node [anchor=north west][inner sep=0.75pt]    {$\approx v( a-b)$};

\end{tikzpicture}
    \caption{Meet points are valuations}
    \label{fig:enter-label}
\end{figure}

\begin{exercise}\label{exo:lemmaHahnseries}
    Let $A,B$ be well-ordered subsets of $\Gamma$. Prove the following:
    \begin{enumerate}
        \item $A\cup B$ is well-ordered;
        \item $A+B = \set{\alpha+\beta\mid\alpha\in A,\beta\in B}$ is well-ordered;
        \item For each $\gamma\in \Gamma$ there are only finitely many pairs $(\alpha,\beta)\in A\times B$ such that $\gamma = \alpha+\beta$.
    \end{enumerate}
\end{exercise}

\begin{exercise}\label{exo:lemmaHahnseries2}
    We prove that $K = k((t^\Gamma))$ is a field. We assume the following:\\
    \textit{(Neumann's Lemma)} Let $A$ be a well-ordered subset of $\Gamma$. Then 
        \[\set{\alpha_1+\ldots+\alpha_n\mid \alpha_i\in A, n\in \N}\]
        is well-ordered and for all $\gamma\in \Gamma$ there are only a finite number of elements of $A$ whose sum equals $\Gamma$.
    \begin{enumerate}
        \item Let $f\in K$ with $v(f)>0$. Prove that $\sum_{n = 0}^\infty f^n\in K$ and that $(1-f)\sum_{n = 0}^\infty f^n = 1$.
        \item Prove that for any $g\in K\setminus \set{0}$ there exists $c\in k$, $\gamma\in \Gamma$ and $f$ with $v(f)>0$ such that $g = ct^\gamma(1-f)$. 
        \item Conclude.
    \end{enumerate}
\end{exercise}

\section{The theorem of Pas}\label{sec:pastheorem}

\subsection{Angular component map} In a valued field $(K,v)$ note that $\res_{\upharpoonright\cO^\times}: \cO^\times\to k^\times$ is a multiplicative group homomorphism. An angular component map on $(K,v)$ is an extension of this homomorphism to the supergroup $K^\times$ of $\cO^\times$.

\begin{definition}
    Given a valued field $(K,v)$ with residue field $k$. An \textit{angular component map} is a map $\ac: K\to k$ such that
    \[\begin{cases}
        \ac(a) = 0 \iff a = 0\\
        \ac:K^\times\to k^\times \text{ is a multiplicative group homomorphism}\\
        \ac(a) = \res(a) \text{ whenever $v(a) = 0$}
    \end{cases}\]
    A valued field equipped with an angular component map is called an \textit{ac-valued field}.
\end{definition}
The first two conditions imply that $\ac$ is a \textit{multiplicative map} i.e. $\ac(ab) = \ac(a)\ac(b)$ for all $a,b\in K$.

\begin{example}
    Main examples of angular component maps:
    \begin{itemize}
        \item (In $\Q_p$.) Let  $f = \sum_{i\geq i_0} a_ip^i$ with $a_{i_0}\neq 0$. Then we define $\ac(f) = a_{i_0}= a_{v_p(f)}$.
        \item (In $k((t^\Gamma))$.) Let  $f = \sum_{i\geq i_0} a_it^i$ with $a_{i_0}\neq 0$. Then we define $\ac(f) = a_{i_0}= a_{v_t(f)}$.
    \end{itemize}
\end{example}

There are examples\footnote{An example is given in \cite{Pas90}.} of valued field which do not have an angular component map, however every valued field has an elementary extension with an angular component map.

Recall that an abelian group $A$ is \textit{pure injective} if for all abelian groups $B,C$ where $B$ is a pure subgroup in $C$ (i.e. $B = \set{c\in C\mid c^n\in B \text{ for some $n\in \N$}}$) any homomorphism $B\to A$ extends to a homomorphism $C\to A$. The following is a classical fact in model theory of groups, see e.g. \cite[Theorem 20, p. 171]{cherlin}.
\begin{fact}\label{fact:cherlin}
    Every $\aleph_1$-saturated abelian group is pure injective.
\end{fact}
\begin{proposition}\label{prop:acsaturated}
    Let $(K,k,\Gamma)$ be an $\aleph_1$-saturated valued field. Then there exists an angular component map $\ac:K\to k$.
\end{proposition}
\begin{proof}
    We want to find an extension of $\res_{\upharpoonright \cO^\times} : \cO^\times\to k^\times$ to $K^\times$. By $\aleph_1$-saturation of $(K,k,\Gamma)$ (actually of the group $(k^\times,\cdot)$) and Fact \ref{fact:cherlin} it suffices to prove that $\cO^\times$ is pure in $K^\times$. If $a\in K^\times$ is such that $a^n\in \cO^\times$, then $v(a^n) = nv(a) = 0$ hence $v(a) = 0$ i.e. $a\in \cO^\times$.
\end{proof}

\subsection{The language of Denef-Pas and Pas Theorem}

We introduce the three-sorted language of Denef and Pas to deal  with ac-valued fields

\begin{definition}
    Let $\LLdp$ be the three sorted language defined by:
\begin{itemize}
    \item one sort for the valued field $K$ in the language of rings $\LLvf = \set{+,-,\cdot,0,1}$ (the \textit{valued field sort})
    \item one sort for the residue field $k$ in a different copy $\LLres = \set{+,-,\cdot,0,1}$ of the language of fields (the \textit{residu field sort})
    \item one sort for the value group $\Gamma$ in the language of ordered groups expanded by a constant $\LLgp = \set{+,-<,0,\infty}$ (the \textit{value group sort})
    \item a function symbol $v:K\to \Gamma\cup \set{\infty}$ for the valuation
    \item a function symbol $\ac: K\to k$ for the angular component map.
\end{itemize}
Any $\LLdp$-structure is given by a tuple $(K,k,\Gamma,v,\ac)$ with the following maps between the three sorts:
\begin{center} 
\begin{tikzcd}
K \ar[r,"v"] \ar[d,"\ac"] & \Gamma\cup\set{\infty}\\
k & 
\end{tikzcd}\\
\end{center}
\end{definition}

Note that the language $\LLdp$ is countable, in the sense that the number of $\LLdp$-formulas is countable.

\begin{definition}
    Let $T_0^\DP$ be the $\LL_\DP$-theory expressing the following for any model $(K,k,\Gamma,v,\ac)$:
\begin{itemize}
    \item $(K,v)$ is a valued field with value group $\Gamma$ (i.e. $v(K^\times) = \Gamma$)
    \item $(K,v)$ is Henselian of equicharacteristic $(0,0)$.
    \item $\ac:K\to k$ is an angular component map for the valued field $(K,v)$ (i.e. $\ac:K^\times\to k^\times$ is a group homomorphism, $\ac(a) = 0$ iff $a = 0$) and the residue map $\res:\cO\to k$ associated to $v$ is onto and coincide with $\ac$ on the set $\cO^\times = \set{a\in K\mid v(a) = 0}$\footnote{To express this: define $\res:\cO \to k$ by cases:
    \[\res(a) = \begin{cases} \ac(a) & \text{if } v(a) = 0\\ 0 & \text{if }v(a)>0\end{cases}\]
    and ask that $\res$ is a ring homomorphism which is surjective and with kernel $\cM$.}).
\end{itemize}
\end{definition}

\begin{definition}
    Let $\Tres$ be a theory of fields in $\LLres$ and $\Tgp$ a theory of ordered abelian group in $\LLgp$ we define $T^\DP = T^\DP(\Tres,\Tgp)$ to be the expansion of $T^\DP_0$ obtained by adding $\Tres$ in $\LLres$ to the residue field sort and $\Tgp$ in $\LLgp$ to the value group sort.
\end{definition}

For a field $k$ and an ordered abelian group $\Gamma$, we will also consider $T^\DP = T^\DP(\Th(k),\Th(\Gamma))$.

In 1989, Johan Pas \cite{Pas89} proves :
\begin{theorem}[Johan Pas]\label{thm:pas}
    For any complete theory $\Tres$ of field in $\LLres$ and for any complete theory $\Tgp$ of ordered abelian group in $\LLgp$, the theory $\TDP = \TDP(\Tres,\Tgp)$ is complete and eliminates the fields quantifiers. This means that for any $\LLdp$-formula $\phi(x,\xi, \bar u)$ there exist an $\LLdp$-formula $\psi(x,\xi,\bar u)$ where the quantifiers $\forall,\exists$ are only over variables from $\LLres$ and $\LLgp$, such that 
    \[\TDP \models \forall x \xi\bar u \left[\phi(x,\xi,\bar u)\leftrightarrow \psi(x,\xi,\bar u) \right].\]
\end{theorem}

\begin{proof}[Proof of the AKE Theorem \ref{thm:AKE} from Theorem \ref{thm:pas}]
    Let $(K,k_K,\Gamma_K)$ and $(L,k_L,\Gamma_L)$ be two valued fields in the three-sorted language $\LLts$ which are Henselian and of equicharacteristic $0$. We need to prove that $(K,k_K,\Gamma_K)\equiv (L,k_L,\Gamma_L)$ in $\LLts$ if and only if
$k_K\equiv k_L$ in $\LLr$ and $\Gamma_K\equiv \Gamma_L$ in $\LLgp$. The 'only if' direction is clear. We prove the 'if' direction. Assume that $k_K\equiv k_L$ and $\Gamma_K\equiv \Gamma_L$. First, consider $(K^*,k_K^*,\Gamma_K^*)$ and $(L^*,k_L^*,\Gamma_L^*)$ two $\aleph_1$-saturated elementary extensions (as $\LLts$ valued fields) of $(K,k_K,\Gamma_K)$ and $(L,k_L,\Gamma_L)$ respectively. By Proposition \ref{prop:acsaturated}, there exists angular component maps $\ac_{K^*}:K^*\to k_K^*$ and $\ac_{L^*}:L^*\to k_L^*$ so that we may consider $(K^*,k_K^*,\Gamma_K^*)$ and $(L^*,k_L^*,\Gamma_L^*)$ as $\LLdp$-structures. Note that $\res$ is always onto the residue field so that $\ac_{K^*}$ and $\ac_{L^*}$ are onto. It follows from the hypotheses that $(K^*,k_K^*,\Gamma_K^*)$ and $(L^*,k_L^*,\Gamma_L^*)$ are models of $\TDP$ for $\TDP = \TDP(\Th(k_K),\Th(\Gamma_K))$. By Theorem \ref{thm:pas}, $\TDP$ is complete, hence $(K^*,k_K^*,\Gamma_K^*)\equiv (L^*,k_L^*,\Gamma_L^*)$ as $\LLdp$ valued fields. In particular, $(K^*,k_K^*,\Gamma_K^*)\equiv (L^*,k_L^*,\Gamma_L^*)$ as $\LLts$ valued fields. As $(K,k_K,\Gamma_K)\equiv (K^*,k_K^*,\Gamma_K^*)$ and $(L,k_L,\Gamma_L)\equiv (L^*,k_L^*,\Gamma_L^*)$ as $\LLts$ valued fields, we conclude $(K,k_K,\Gamma_K)\equiv (L,k_L,\Gamma_L)$.
\end{proof}

\section{Proof of Pas theorem}

\subsection{Algebraic preliminaries}

Recall that a valuation satisfies $v(a+b)\geq \min\set{v(a),v(b)}$. The ambiguity really comes when $v(a) = v(b)$, since we have the following:
\[v(a)<v(b)\implies v(a+b) = v(a)\]
Indeed, assume that $v(a)<v(b)$, then $v(a) = v(a+b - b)\geq \min \set{v(a+b),v(b)}$ (as $v(b) = v(-b)$ by Exercise \ref{exo:valdebase1}). Since $v(a)<v(b)$ it must be that $v(a)\geq v(a+b)$. On the other hand we have $v(a+b)\geq \min\set{v(a),v(b)} = v(a)$

\begin{exercise}\label{exo:valdebase1}
Let $(K,v)$ be a valued field.
    Prove the following:
    \begin{enumerate}
        \item $v(1) = v(-1) = 0$
        \item $v(a) = v(-a)$
    \item If $v(a_1+\ldots+a_n)>\min\set{v(a_i)}$ then there exists $i\neq j$ such that $v(a_i) = v(a_j)$.
    \item If $(a_1,\ldots,a_n)\neq \vec 0$ and $\sum_{i = 0}^n a_i = 0$ then there exists $i\neq j$ such that $v(a_i) = v(a_j)$.
\end{enumerate}
\end{exercise}

The following is a key lemma to understand how valuations extend to field extensions.

\begin{lemma}
Let $(L,w)$ be an extension of the valued field $(K,v)$. Let 
\begin{itemize}
    \item $a_1,\ldots,a_r\in \cO_w$ such that $\res(a_1),\ldots, \res(a_r)\in k_L$ are $k_K$-linearly independent;
    \item $b_1,\ldots, b_s\in L^\times$ such that $w(b_1),\ldots,w(b_s)$ are in different classes modulo $\Gamma_K$;
    \item $\set{0}\neq \set{c_{i,j}\mid 1\leq i\leq r, 1\leq j\leq s}\seq K$
\end{itemize}
Then 
\[w(\sum_{i,j} c_{i,j}a_ib_j) = \min_{i,j} \set{w(c_{i,j}a_ib_j)} = \min_{i,j}\set{v(c_{i,j})+w(b_j)}\]
In particular, $(a_ib_j)_{i,j}$ are $K$-linearly independent.
\end{lemma}
\begin{proof}
    First, as $\res(a_i)$ are nonzero, we have $v(a_i) = 0$ for all $i\leq r$. Let $\gamma = \min\set{v(c_{i,j})+w(b_j)\mid i,j}$ and $I = \set{(i,j)\mid v(c_{i,j})+w(b_j) = \gamma}$. As $w(\sum_{(i,j)\notin I}c_{i,j}a_ib_j)\geq \min\set{v(c_{i,j})+w(b_j)\mid (i,j)\notin I}>\gamma$, it is enough to show that $w(\sum_{(i,j)\in I} c_{i,j}a_ib_j) = \gamma$. Observe that there exists $j_0$ such that for all $(i,j)\in I$ we have $j = j_0$: otherwise $v(c_{i,j})+v(b_j) = v(c_{i',j'})+w(b_{j'})$ for $j\neq j'$ which contradicts that $w(b_j)$ and $w(b_j')$ are in different cosets modulo $\Gamma_K$. In particular, $v(c_{i,j}) = v(c_{i',j_0})$. Fix $(i_0,j_0)\in I$. Then 
    \[ \frac{1}{c_{i_0,j}b_j} \sum_{(i,j)\in I} c_{i,j}a_ib_j = \sum_{(i,j)\in I} \frac{c_{i,j}}{c_{i_0,j}} a_i = u\]
    It remains to prove that $v(u) = 0$ or equivalently $\res(u)\neq 0$ (as $v(c_{i,j}/c_{i_0,j})\geq 0$) which follows from $\res(a_1)\ldots \res(a_r)$ being linearly independent over $k_K$. The `in particular' part is immediate: $b_i$ are nonzero hence $w(b_i)\neq \infty$ so if $\sum_{i,j} c_{i,j}a_ib_j = 0$ with $(c_{i,j})_{i,j}\neq \vec 0$ then valuation is $\infty$ and so is $\min\set{v(c_{i,j})+v(b_j)}$, a contradiction.
\end{proof}

\begin{remark}
    Let $(K,v)\seq (L,w)$ be a valued fields extension with $L$ a finite field extension of $K$. Then 
    \[[L:K]\geq [k_L:k_K][\Gamma_L:\Gamma_K]\]
    To see this, take $(a_i)_{i\in I}$ such that $(\res(a_i))$ are $k_K$-independent and $(b_j)_{j\in J}$ with $(v(b_j))$ in different cosets modulo $\Gamma_K$, then $\set{a_ib_j\mid i,j}$ are $K$-linearly independent so that $[L:K]\geq \abs{\set{a_ib_j\mid i,j}} = \abs{I\times J}$.
\end{remark}

 When we consider a valued field extension $(K,v)\seq (L,w)$, we have that $w\upharpoonright K = v$ hence for now on we will write $(K,v)\seq (L,v)$.

\begin{corollary}\label{cor:3.7bestextension}
    Let $(K,v)\seq (L,v)$ be a valued fields extension.
    \begin{enumerate}
        \item Let $a\in L$ be such that $1,\res(a),\ldots, \res(a^{n})$ are linearly independent over $k_K$. Then for all $c_0,\ldots,c_n\in K$ we have 
        \[v(\sum_i c_ia^i) = \min_i\set{v(c_i)}\]
        In particular $v(K+Ka+\ldots+Ka^n)\seq \Gamma_K\cup\set{\infty}$.
        \item Let $a\in L$ be such that $0,v(a),\ldots, v(a^n)\in \Gamma_L$ are in different classes modulo $\Gamma_K$. Then for all $c_0,\ldots,c_n\in K$ we have 
        \[v(\sum_ic_i a^i) = \min_i\set{v(c_i)+iv(a)}\]
        In particular $v(K+Ka+\ldots+Ka^n)\seq \vect{\Gamma_K,v(a)}\cup\set{\infty}$.
    \end{enumerate}
\end{corollary}

\begin{remark}\label{rk:productofelementsunramified}
Let $(K,v)\seq (L,v)$ such that $\Gamma_K = \Gamma_L$ then for all $a\in L$ there exist $b\in K$ and $c\in \cO_L^\times$ such that $a = bc$. Indeed, let $a\in L$, we have $v(a)\in \Gamma_K$ hence there exists $b\in K$ such that $v(a) = v(b)$. Then for $c = ab^{-1}$ we have $v(c) = 0$ so $c\in \cO_L^\times$ and $a = bc$.
\end{remark}



\begin{remark}\label{rk:productofthreeelementsramifiedcase}
    Let $(K,v)\seq (L,v)$ be such that $\Gamma_L = \vect{\Gamma_K, \alpha}$ for some element $\alpha$. Assume that $a\in L$ is such that $\alpha = v(a)$. Then every element of $L$ is a product $bca^n$ where $b\in K$, $c\in \cO_L^{\times}$ for some $n\in \Z$.
    Indeed: if $e\in L$ we have $v(e) = \gamma+n\alpha$ for some $\gamma\in \Gamma_K$ and $n\in \Z$. Then $v(ea^{-n}) = \gamma\in \Gamma_K$ hence there exists $b\in K$ such that $v(b) = \gamma$ hence for $c = eb^{-1}a^{-n}$ we have $v(c) = 0$ i.e. $c\in \cO_L^\times$ and $e = bca^n$.
\end{remark}

\subsubsection{Henselian fields} Recall that a valued field $(K,v)$ is Henselian if it satisfies the following property:

    \textbf{Simple zero lift}. For each $P\in \cO[X]$ and $\bar a\in k$ such that $\res(P)(\bar a) = 0$ and $\res(P')(\bar a) \neq  0$ there exists $b\in \cO$ such that $P(b) = 0$ and $\res(b) = \bar a$.

\begin{lemma}\label{lm:henselemazero}
    Let $(K,v)$ be Henselian of equicharacteristic $0$. If $P(X)\in \cO_K[X]$ is such that $v(P(0))>2v(P'(0))$, then there exists $a\in \cO_K$ such that 
    \[\begin{cases}
        P(a) = 0\\
       v(a) = v(P(0))-v(P'(0)).
    \end{cases}\]
\end{lemma}
\begin{proof}
    Write $P(X) = a_0+a_1X+\ldots+a_nX^n$ and $P'(X) = a_1+2a_2X+\ldots+na_nX^{n-1}$, so that $a_0 = P(0)$ and $a_1 = P'(0)$. Let $Q(X) = \frac{1}{a_0}P(cX)$ for $c = -\frac{a_0}{a_1}$, then we have
    \[Q(X) = 1-X +\sum_{i\geq 2} \frac{a_i}{a_0} c^i X^i\]
    Note that 
    \begin{align*}
        v(c^ia_0^{-1}) &= (i-1)v(a_0)-iv(a_1)\\
        &= (i-1)(v(a_0)-\frac{i}{i-1} v(a_1))\\
        &\geq (i-1)(v(a_0)-2v(a_1))>0
    \end{align*} 
    In particular $v(\frac{a_i}{a_0} c^i) = v(a_i)+v(c^ia_0^{-1}) >0$. It follows that $v(Q(1))>0$ hence $\res(Q)(1) = 0$ and $\res(Q')(1) = 1$. By the simple zero lift, there exists $b\in K$ such that $Q(b) = 0$ and $\res(b) = 1$. In particular $v(b) = 0$. Let $a = cb$, we have $P(a) = 0$ and $v(a) = v(c) = v(P(0))-v(P'(0))$.
\end{proof}

\begin{corollary}\label{cor:henseliannthroots}
    Let $(L,v)$ be a Henselian valued field of equicharacteristic $0$. If $(K,v)\seq (L,v)$ is such that $k_K = k_L$ and $\gamma\in \Gamma_L$ is such that $n\gamma\in \Gamma_K$, then there exists $a\in L$ such that 
        \[\begin{cases}
            a^n\in K\\
            v(a) = \gamma
        \end{cases}\]
\end{corollary}
\begin{proof}
    We first establish the following.
    \begin{claim}\label{claim_crux}
        For all $b\in \cM_L$ and for all $n\in \N$ there exists $a\in L$ such that 
        \[\begin{cases}
            a^n = 1+b\\
            v(a-1) = v(b)
        \end{cases}\]
    \end{claim}
    \begin{proof}[Proof of the claim]
    Let $P(X) = (X+1)^n - (1+b)$. We have $v(P(0)) = v(b)>0 = 2v(1) = 2P'(0)$. By Lemma \ref{lm:henselemazero}, there exists $c\in L$ such that $P(c) = 0$ and $v(c) = v(P(0))-v(P'(0)) = v(b)$. Then $a = c+1$ is suitable for the claim.
    \end{proof}
    Let $b\in L$ and $c\in K$ such that $v(b) = \gamma$ and $v(c) = n\gamma$. We have $v(b^nc^{-1}) = 0$ so we may apply $\res$ and as $k_K = k_L$, there exists $d\in \cO_K^\times$ such that $\res(d) = \res(b^nc^{-1})$ ($d$ is of valuation $0$ since otherwise $\res(d) = 0$). We set $c' = cd$. Then we have $\res(b^nc'^{-1}) = 1$ so $b^n c'^{-1} = 1+u$ for $u\in \cM_L$ and $b^n = c'(1+u)$. By the claim, $1+u$ has an $n$-th root $e$ with $v(e-1) = v(u)>0$, i.e. $\res(e-1) = 0$ so $\res(e) = \res(1)$ so $v(e) = 0$. It follows that $v(be^{-1}) = v(b) = \gamma$ and $(be^{-1})^n = c'$. This finishes the proof with $a = be^{-1}$.
\end{proof}

We finish with two important and classical theorems on Henselian fields.
\begin{theorem}\label{thm:henselianalgebraicextension}
    $(K,v)$ is Henselian if and only if for all algebraic field extension $L$ of $K$, there exists a unique valuation $w$ on $L$ which extends $v$.
\end{theorem}

\begin{theorem}\label{thm:ostrowski}(Ostrowski)
    Let $(K,v)$ be a Henselian valued field and $L$ a finite field extension of $K$. Let $w$ be the unique extension of $v$ to $L$. Then
    \[[L:K] = [k_L:k_K][\Gamma_L:\Gamma_K]\chi^d\]
    for some $d\in \N$ and 
    \[
    \chi = \begin{cases}
        \mathrm{char}(k_K) &\text{ if $\mathrm{char}(k_K)>0$}\\
        1 &\text{ if $\mathrm{char}(k_K) = 0$.}
        \end{cases}
    \]
\end{theorem}

\begin{remark}
    The number $d$ in Theorem \ref{thm:ostrowski} is usually called the \textit{defect} of the extension $L/K$.
\end{remark}

The proofs of Theorem \ref{thm:henselianalgebraicextension} and \ref{thm:ostrowski} are beyond the scope of this course.

\begin{corollary}\label{cor:henselianchar0algmax}
    Let $(K,v)$ be an Henselian valued field with $\mathrm{char}(k) = 0$, then $K$ has no proper immediate algebraic extensions (i.e. an extension $(L,w)$ of $(K,v)$ such that $k_K = k_L $ and $\Gamma_K = \Gamma_L$).
\end{corollary}
\begin{proof}
    If $(K,v)$ has an immediate algebraic extension, it has an immediate finite extension $(L,w)$. As $\Gamma_K = \Gamma_L$ and $k_K = k_L$ and $\chi = 1$ (as $\mathrm{char}(k_K) = 0$) it remains $[L:K] = 1$ by Ostrowski's theorem.
\end{proof}

\begin{exercise}\label{exo:henselslemma}
    Prove that the simple zero lift property is equivalent to
    \begin{center}
    $(\ast)$ given $a\in \cO$ and $P\in \cO[X]$ with $v(P(a))>2v(P'(a))$, there exists $b\in \cO$ with $P(b) = 0$ and $v(b-a)>v(P'(a)))$
    \end{center}
    First, assume that $v(P(a))>2v(P'(a))$ for some $a\in \cO$ and $P\in \cO[X]\setminus\set{0}$.
    \begin{enumerate}
        \item Prove that $P'(a)\neq 0
        $. 
        \item Prove that there exists $Q(Y,X)\in K[X,Y]$ such that $P(a-X) = P(a)-P'(a)X+X^2Q(a,X)$ (\textit{Hint:} Check out Lemma \ref{lm:taylor} below).
        \item Prove that for $Y = X/P'(a)$ the polynomial
        \[R(Y) := \frac{P(a-P'(a)Y}{P'(a)^2}\]
        satisfies:
        \begin{enumerate}
            \item $R(Y)\in \cO[Y]$.
            \item $\res(R)(0) = 0$, $\res(R')(0) = -1$.
        \end{enumerate}
        \item Use the simple zero lift property prove that there exists $c\in \cO$ such that $R(c) = 0$.
        \item Conclude $(\ast)$ by taking $b = a-f'(a)c\in \cO$.
    \end{enumerate}
    Conversely, if $P\in \cO[X]$, $\res(P)(\bar a) = 0\neq \res(P')(\bar a)$ then for any lift $a\in \cO$ of $\bar a$ we have $v(P(a)) >0 = v(P'(a))$. Conclude using $(\ast)$.
\end{exercise}

\begin{exercise}\label{exo:definitionZpdansQP}
    We prove that $\Z_p$ is definable in $\Q_p$ in the language of rings:
    \[\Z_p = \set{a\in \Q_p\mid  \exists y(  1+pa^2 = y^2)} \text{ if $p\neq 2$}\]
    and 
    \[\Z_2 = \set{x\in\Q_p\mid \exists y(  1+2x^3 = y^3)}.\]
    We detail the steps for $p\neq 2$, the case $p=2$ is similar. 
    \begin{enumerate}
        \item If $a\in \Q_p\setminus \Z_p$.
        \begin{enumerate}
            \item Check that $v(a)$ is even if $a$ is a square (this does not use $a\notin \Z_p$).
            \item Prove that $v(pa^2)\leq -1$.
            \item Deduce that $v(1+pa^2)\in \Z$ is odd.
            \item Conclude.
        \end{enumerate} 
        \item If $a\in \Z_p$, consider $P(Y) = Y^2-(1+pa^2)$.
        \begin{enumerate}
            \item Prove that $v(P(1))>2v(P'(1))$.
            \item Conclude using Exercise \ref{exo:henselslemma}.
        \end{enumerate}
    \end{enumerate}
\end{exercise}

\begin{exercise}
    Let $p>2$. Prove that in $\F_p((t))$ the ring $\F_p[[t]]$ is definable with the parameter $t$ by the formula \[\exists y\ 1+tx^2 = y^2.\]
    (\textit{Hint:} Proceed as in Exercise \ref{exo:definitionZpdansQP}.)
\end{exercise}

\subsubsection{Henselization}

\begin{fact}\label{fact:henselization}
    Let $(K,v)$ be any valued field. There exists a valued field extension $(K^h,v^h)$ of $(K,v)$ such that:
    \begin{enumerate}
        \item $K^h$ is an algebraic extension of $K$ (as fields), i.e. $K\seq K^h\seq K^\alg$;
        \item $(K^h,v^h)$ is Henselien;
        \item If $(L,w)$ is a Henselian valued field extending $(K,v)$, then there exists an embedding of valued fields $i:(K^h,v^h)\to (L,w)$ over $K$ (i.e. a field embedding $i:K^h\to L$ such that $i\upharpoonright K = \Id_K$ and $i(\cO_{K^h}) = \cO_L\cap i(K^h)$).
        \item $(K^h,v^h)$ is an \emph{immediate} extension of $(K,v)$, i.e. $k_K = k_{K^h}$ and $\Gamma_K = \Gamma_{K^h}$.
    \end{enumerate}
    $(K^h,v^h)$ is called the \textit{Henselization of $(K,v)$}.
\end{fact}

$(3)$ will be called the \textit{universal property of the Henselization} and as often with this sort of property, it implies that $(K^h,v^h)$ is unique up to $K$-isomorphism of valued field. The proof of Fact \ref{fact:henselization} is beyond the scope of this course, however we will explain how $(K^h,v^h)$ is constructed using infinite Galois theory.

For convenience we assume that $K$ is of characteristic $0$ but what we will describe now has an equivalent in positive characteristic. The \textit{absolute Galois group} of the field $K$ is by definition:
\[G_K := \Aut(K^\alg/K)\]
In infinite Galois theory, $G_K$ is identified with the inverse limit of the inverse system of finite groups 
\[\set{\mathrm{Gal}(L/K)\mid L \text{ finite Galois extension of $K$}}\]
with the restriction maps $\Gal(M/K)\to \Gal(L/K)$ as connecting homomorphisms, for $K\seq L\seq M$. Essentially, an element $\sigma\in G_K$ is though of as the family
\[(\sigma_L\mid \text{$L$ finite Galois extension of $K$, $\sigma_L\in \Gal(L/K)$ and if $M\supseteq L\supseteq K$ $\sigma_M\upharpoonright L = \sigma_L$}).\]
$G_K$ is thus a \textit{profinite group} (=inverse limit of finite groups) and as such is endowed with a topology, which admits cosets of normal subgroups of finite index as a basis of open sets. The Galois correspondence gives that there is a one-to-one correspondence
\[\set{\text{closed subgroups of $G_K$}}\leftrightarrow \set{\text{intermediate fields $K\seq L\seq K^\alg$}}\]
given by 
\begin{align*}
    H&\mapsto \fix(H) = \set{a\in K^\alg\mid \sigma(a) = a \text{ for all $\sigma\in H$}}\\
    \Gal(K^\alg/L)&\text{ \reflectbox{$\mapsto$} } L.
\end{align*}

This correspondence is of course more precise (e.g. $L/K$ is Galois iff $\Gal(L/K)$ is normal in $G_K$, etc). We now consider the valued field $(K,v)$. We will use two standard facts from classical valuation theory:
\begin{enumerate}[label=\alph*)]
    \item \textit{(Extension Theorem)} For any field extension $L$ of $K$ there exists a valuation $w$ on $L$ extending $v$.
    \item \textit{(Conjugation Theorem)} If $L$ is a normal field extension of $K$ and $w_1,w_2$ two valuations on $L$ extending $v$, then there exists a field automorphism $\sigma$ of $L$ over $K$ such that $\sigma(\cO_{w_1}) = \cO_{w_2}$.
\end{enumerate}
By the extension theorem, there exists a valuation $w$ on $K^\alg$ extending the valuation $v$ on $K$. We define
\[D_w := \set{\sigma\in G_K\mid \sigma(\cO_w) = \cO_w}\seq G_K\]
Note that $D_w$ is the automorphism group of the valued field $(K^\alg,w)$ over $(K,v)$. One proves that $D_w$ is a closed subgroup of $G_K$ and that for any other extension $w'$ of $v$ to $K^\alg$ the groups $D_w$ and $D_{w'}$ are conjugate as subgroups of $G_K$ (in particular $D_w$ may not be a normal subgroup). We can now define $(K^h,v^h)$:
\[K^h := \fix(D_w)\quad; \quad v^h := w\upharpoonright K^h.\]
This already gives $(1)$ of Fact \ref{fact:henselization}. Using the Galois correspondence, we have 
\[\Aut(K^\alg/K^h) = D_w = \set{\sigma\in G_K\mid \sigma(\cO_w) = \cO_w}.\]
By the Conjugation Theorem, any extension of $v^h$ to $K^\alg$ have to be conjugated by an element of $\Aut(K^\alg/K^h)$ hence $v^h$ has a unique extension to $K^\alg = (K^h)^\alg$. By Theorem \ref{thm:henselianalgebraicextension}, this gives that $(K^h,v^h)$ is Henselian (2). (3) and (4) need more work, see e.g. \cite{jahnkecourse}.

\subsubsection{Kaplanski theory of pseudo-convergence}\label{subsub:kaplanski} In this section we consider sequences of elements in a valued field $(K,v)$ with value group $\Gamma$. Most sequences will be indexed by ordinals. Those results are due to Kaplanski \cite{kaplanskimaximalvalued} and are classical.

\begin{definition}
    Let $(a_i) = (a_i)_{i<\lambda}$ be a sequence in $(K,v)$ for some limit ordinal $\lambda$. \begin{enumerate}
        \item We say that $(a_i)$ \textit{pseudoconverges} to $a\in K$, denoted $(a_i)\rightsquigarrow a$ if $(v(a_i-a))_{i<\lambda}$ is eventually strictly increasing, i.e. there exists $i_0<\lambda$ such that for all $i_0<i<j<\lambda$ we have \[v(a_j-a)> v(a_i-a)\]
        We say that $a$ is a \textit{pseudolimit} of $(a_i)$.
        \item $(a_i)$ is a \textit{pseudo-Cauchy sequence} if there exists $i_0<\lambda$ such that for all $i_0<j_1<j_2<j_3<\lambda$ we have 
        \[v(a_{j_3}-a_{j_2})>v(a_{j_2} - a_{j_1})\]
    \end{enumerate}
\end{definition}

\begin{remark}\label{rk:pc-sequences}
    Some easy facts.
    \begin{enumerate}
        \item (A pseudolimit is rarely unique). In fact, if $(a_i)\rsa a$ then for all $b$ we have $(a_i)\rsa b$ if and only if $v(a-b)> v(a-a_i)$ eventually (i.e. there exists $i_0<\lambda$ such that $v(a-b)>v(a-a_i)$). See Exercise \ref{exo:differentlimits}.
        \item (Every pseudoconvergent sequence is a pseudo-Cauchy sequence). If $(a_i)\rsa a$ then $(a_i)$ is a pc-sequence: let $i_0<\lambda$ be such that $v(a-a_j)>v(a-a_i)$ for all $i_0<i<j$, then if $i_0<j_1<j_2<j_3$ we have $v(a_{j_3} - a_{j_2})= v(a_{j_3} - a + a-a_{j_2}) = v(a-a_{j_2})$ because $v(a-a_{j_3})>v(a-a_{j_2})$. Similarly $v(a_{j_2}-a_{j_1}) = v(a-a_{j_1})$, hence as $v(a-a_{j_2})> v(a-a_{j_1})$ we conclude:
        \[v(a_{j_3}-a_{j_2}) = v(a-a_{j_2})> v(a-a_{j_1}) = v(a_{j_2} - a_{j_1})\]
        \item (Valuation of a pc-sequence, I) If $(a_i)$ is a pc-sequence then we will consider the sequence $(\alpha_i)\seq \Gamma$ such that $v(a_{i+1}-a_i) = \alpha_i$. The sequence $(\alpha_i)$ is eventually strictly increasing. Indeed, for all $j>i>i_0$ we have 
        $v(a_{i+1}-a_i) = v(a_{i+1} - a_j +a_j-a_i) = v(a_j-a_i)$ since $v(a_{i+1}-a_j)>v(a_j-a_i)$. Also $\alpha_{i+1} = v(a_{i+2}-a_{i+1})>v(a_{i+1}-a_i) = \alpha_i$ for $i>i_0$. If $(a_i)\rsa a$ we also have $\alpha_i = v(a - a_i)$ eventually.
        \item (Valuation of a convergent sequence) If $(a_i)\rsa a$ then the sequence $(\beta_i) = (v(a_i))$ is eventually strictly increasing or eventually constant. Indeed, suppose first that $v(a_i)\geq v(a)$ for some $i>i_0$, then for all $j>i$ we have $v(a-a_j)>v(a-a_i)\geq\min\set{v(a_i),v(a)} = v(a)$ hence $v(a) = v(a_j)$ so $(\beta_i)$ is eventually constant. Otherwise, $v(a_i)<v(a)$ for all $i> i_0$ and for $i_0<i<j$ we have $v(a_i) = v(a-a_i)<v(a-a_j) = v(a_j)$.
    \end{enumerate}
\end{remark}

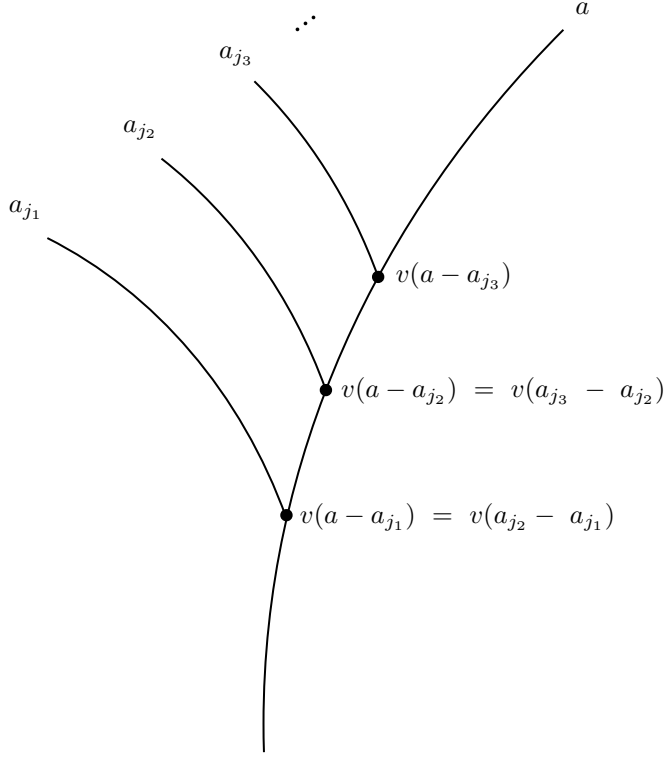
\begin{figure}
    \centering
\tikzset{every picture/.style={line width=0.75pt}} 

\begin{tikzpicture}[x=0.75pt,y=0.75pt,yscale=-1,xscale=1]

\draw  [draw opacity=0] (303.72,406.96) .. controls (303.57,402.15) and (303.5,397.33) .. (303.5,392.5) .. controls (303.5,259.8) and (359.72,138.16) .. (453.19,43.59) -- (950.75,392.5) -- cycle ; \draw   (303.72,406.96) .. controls (303.57,402.15) and (303.5,397.33) .. (303.5,392.5) .. controls (303.5,259.8) and (359.72,138.16) .. (453.19,43.59) ;  
\draw  [draw opacity=0] (252.56,108.4) .. controls (288.88,137.07) and (317.63,177.47) .. (334.65,224.83) -- (123.08,324.83) -- cycle ; \draw   (252.56,108.4) .. controls (288.88,137.07) and (317.63,177.47) .. (334.65,224.83) ;  
\draw  [line width=2.25]  (359.14,167.91) .. controls (359.14,167.03) and (359.85,166.32) .. (360.73,166.32) .. controls (361.6,166.32) and (362.31,167.03) .. (362.31,167.91) .. controls (362.31,168.78) and (361.6,169.49) .. (360.73,169.49) .. controls (359.85,169.49) and (359.14,168.78) .. (359.14,167.91) -- cycle ;
\draw  [draw opacity=0] (195.6,148.36) .. controls (249.58,175.78) and (292.58,226.47) .. (314.95,289.34) -- (103.08,388.5) -- cycle ; \draw   (195.6,148.36) .. controls (249.58,175.78) and (292.58,226.47) .. (314.95,289.34) ;  
\draw  [draw opacity=0] (298.94,69.5) .. controls (325.7,96.13) and (347.02,129.76) .. (360.73,167.91) -- (149.16,267.91) -- cycle ; \draw   (298.94,69.5) .. controls (325.7,96.13) and (347.02,129.76) .. (360.73,167.91) ;  
\draw  [line width=2.25]  (333.06,224.83) .. controls (333.06,223.95) and (333.77,223.24) .. (334.65,223.24) .. controls (335.52,223.24) and (336.24,223.95) .. (336.24,224.83) .. controls (336.24,225.7) and (335.52,226.42) .. (334.65,226.42) .. controls (333.77,226.42) and (333.06,225.7) .. (333.06,224.83) -- cycle ;
\draw  [line width=2.25]  (313.36,287.76) .. controls (313.36,286.88) and (314.07,286.17) .. (314.95,286.17) .. controls (315.83,286.17) and (316.54,286.88) .. (316.54,287.76) .. controls (316.54,288.63) and (315.83,289.34) .. (314.95,289.34) .. controls (314.07,289.34) and (313.36,288.63) .. (313.36,287.76) -- cycle ;
\draw   (320.03,43.4) .. controls (320.03,43.17) and (320.22,42.98) .. (320.45,42.98) .. controls (320.68,42.98) and (320.87,43.17) .. (320.87,43.4) .. controls (320.87,43.63) and (320.68,43.82) .. (320.45,43.82) .. controls (320.22,43.82) and (320.03,43.63) .. (320.03,43.4) -- cycle ;
\draw   (328.03,37.4) .. controls (328.03,37.17) and (328.22,36.98) .. (328.45,36.98) .. controls (328.68,36.98) and (328.87,37.17) .. (328.87,37.4) .. controls (328.87,37.63) and (328.68,37.82) .. (328.45,37.82) .. controls (328.22,37.82) and (328.03,37.63) .. (328.03,37.4) -- cycle ;
\draw   (324.03,40.4) .. controls (324.03,40.17) and (324.22,39.98) .. (324.45,39.98) .. controls (324.68,39.98) and (324.87,40.17) .. (324.87,40.4) .. controls (324.87,40.63) and (324.68,40.82) .. (324.45,40.82) .. controls (324.22,40.82) and (324.03,40.63) .. (324.03,40.4) -- cycle ;

\draw (457.67,28.4) node [anchor=north west][inner sep=0.75pt]    {$a$};
\draw (231.87,87.67) node [anchor=north west][inner sep=0.75pt]    {$a_{j_{2}}$};
\draw (368,159) node [anchor=north west][inner sep=0.75pt]    {$v( a-a_{j_{3}})$};
\draw (279.87,50.6) node [anchor=north west][inner sep=0.75pt]    {$a_{j_{3}}$};
\draw (174.8,128.2) node [anchor=north west][inner sep=0.75pt]    {$a_{j_{1}}$};
\draw (340.8,217) node [anchor=north west][inner sep=0.75pt]    {$v( a-a_{j_{2}}) \ =\ v( a_{j_{3}} \ -\ a_{j_{2}})$};
\draw (320.13,280.2) node [anchor=north west][inner sep=0.75pt]    {$v( a-a_{j_{1}}) \ =\ v( a_{j_{2}} -\ a_{j_{1}})$};

\end{tikzpicture}
\caption{A pseudo-convergent sequence is pseudo-Cauchy}
    \label{fig:pseudoconvergentpc}
\end{figure}

We start by constructing limits of pseudo-Cauchy sequences at the cost of extending the valuation.

\begin{lemma}\label{lm:everypcsequenceisconvergent}
    Let $(a_i)_{i<\lambda}$ be a pseudo-Cauchy sequence in $K$ and let $(L,v)$ be a $\abs{\lambda}^+$-saturated elementary extension of $(K,v)$. Then there exists $a\in L$ such that $(a_i)\rsa a$ (in the valued field $(L,v)$).
\end{lemma}
\begin{proof}
    Let $i_0<\lambda$ be as in the definition of $(a_i)_{i<\lambda}$ being a pc-sequence. Consider the set of formulas:
    \[\Delta(x) = \set{v(x-a_j)>v(x-a_i)\mid j>i>i_0}\]
    For any finite subset $\Delta_0(x)$, if $j_0$ is the maximal of the indexes of the $a_i$ appearing in $\Delta_0$, then for any $\lambda> j_1>j_0$ we have 
    \[v(a_{j_1} - a_j)>v(a_{j_1}-a_i)\]
    for all $j_0>j>i>i_0$. Hence $a_{j_0}$ satisfies $\Delta_0(x)$. As $\Delta_0$ was arbitrary, $\Delta(x)$ is finitely consistent. As the cardinality of $\Delta$ is $\abs{\lambda}$, it is satisfied in any $\abs{\lambda}^+$-saturated elementary extension of $(K,v)$.
\end{proof}

\begin{remark}[Valuation of a pc-sequence, II]\label{rk:dichotomyofvaluesofpcsequences}
If $(a_i)$ is a pc-sequence, then $(\beta_i) = v(a_i)$ is either eventually strictly increasing or eventually constant. Indeed: from Lemma \ref{lm:everypcsequenceisconvergent} $(a_i)$ is a pseudoconvergent sequence (in an extension of $(K,v)$), then conclude from Remark \ref{rk:pc-sequences} (4), since $v(a_i)$ live in $\Gamma_K$.
\end{remark}

\begin{lemma}[Formal Taylor expansion]\label{lm:taylor}
Let $P(X)\in K[X]$ of degree $\leq n$, then there exists $P_0,\ldots,P_n$ such that $P(X+Y) = \sum_{i=0}^n P_i(X)Y^i$, with $P_0(X) = P(X)$, $P_1(X) = P'(X)$ and $\deg(P_i)\leq n-i$. Moreover, if $\cO\seq K$ is a subring and $P(X)\in \cO[X]$ then $P_i(X)\in \cO[X]$.
\end{lemma}
\begin{proof}
    This is left as an exercise. Prove that for $P(X) = X^n $ we have $P_i(X) = C^{i}_n X^{n-i}$ and extend to arbitrary $P$ by $K$-linearity. In characteristic $0$, this is the Taylor expansion, $P_i(X) = \frac{P^{(i)}(X)}{i!}$. 
\end{proof}

\begin{theorem}[Polynomials are continuous]\label{thm:polynomialscontinuous}
    Let $(a_i)_{i<\lambda}$ be a sequence of elements in $K$, $a\in K$ and let $P(X)\in K[X]$ be a nonconstant polynomial. If $(a_i)\rsa a$ then $(P(a_i))\rsa P(a)$. In particular if $(a_i)$ is a pc-sequence then $(P(a_i))$ is a pc-sequence.
\end{theorem}

\begin{proof} We start with a claim.
    \begin{claim}
        Let $n\in \N^{>0}$, $\beta_1,\ldots,\beta_n\in \Gamma$, $m_1,\ldots,m_n$ distinct elements of $\N^{>0}$. Let $f_i:\Gamma\to \Gamma$ the function $f_i(\gamma) = \beta_i+m_i\gamma$. Let $(\gamma_j)$ be a strictly increasing sequence indexed by a limit ordinal. Then there exists $i_0\in I$ such that $f_{i_0}(\gamma_j)<f_i(\gamma_j)$ eventually (in $j$) for all $i\in I\setminus \set{i_0}$
    \end{claim}
    \begin{proof}[Proof of the claim]
        This is an easy exercise, by induction on $n$.
    \end{proof}
    Let $P$ be nonconstant of degree $n$, by Lemma \ref{lm:taylor} there exists $(P_i)$ such that
    \[P(X+Y) = P(X)+P_1(X)Y+\ldots+P_n(X)Y^n\]
    in $K[X,Y]$. Substitute $X$ with $a$ and $Y$ with $a-a_i$ we get:
    \[P(a_i)-P(a) = \sum_{i = 1}^n P_i(a)(a_i-a)^i.\]
    Similarly $P(X) = P(a+(X-a)) = P(a)+\sum_{i = 1}^n P_i(a)(X-a)^i$, hence as $P$ is nonconstant, $P_{i_0}(a)\neq 0$ for some $1\leq i_0\leq n$. Let $\beta_i = v(P_{i_0}(a))$ and $\gamma_j = v(a_j-a)$, we have $v(P_{i_0}(a)(a_{j}-a)^{i_0}) = \beta_{i_0}+i\gamma_j$. By the claim there exists $1\leq j_0\leq n$ such that for every $1\leq j\leq n$ with $j\neq j_0$, we have $\beta_{j_0}+j_0 \gamma_i< \beta_{j}+j \gamma_i $ eventually (in $i$). Thus $v(P(a_i)-P(a))=\beta_{j_0}+j_0 \gamma_i$ eventually. As $(\gamma_i)$ is eventually strictly increasing (since $(a_i)\rsa a$), so is $v(P(a_i)-P(a))$ hence $P(a_i)\rsa P(a)$.
\end{proof}

\begin{remark}
    Let $(a_i)$ be a pc-sequence in $K$ and let $P\in K[X]\setminus \set{0}$. By Theorem \ref{thm:polynomialscontinuous}, $(P(a_i))$ is a pc-sequence hence by Remark \ref{rk:dichotomyofvaluesofpcsequences} the sequence $(v(P(a_i))$ is either eventually strictly increasing or eventually constant.
\end{remark}

\begin{definition}[Algebraic type, transcendental type]
    A pc-sequence $(a_i)$ in $K$ is \textit{of transcendental type over $K$} if for all $P\in K[X]\setminus\set{0}$ the sequence $v(P(a_i))$ is eventually constant. Otherwise $(a_i)$ is of \textit{algebraic type} over $K$.
\end{definition}


\begin{remark}\label{rk:transcendentaltype12}
    If $(a_i)$ is a pc-sequence of transcendental type over $K$, then:
    \begin{enumerate}
        \item $(a_i)$ has no pseudolimit in $K$. Indeed, if $(a_i)\rsa a\in K$, then consider $X-a\in K[X]\setminus \set{0}$ to reach a contradiction.
        \item the eventual valuation of $(P(a_i))$ is never $\infty$. If this happens, then $(P(a_i))$ is eventually constant equal to $0$ but such sequence is not pseudo-Cauchy, contradicting Theorem \ref{thm:polynomialscontinuous}.
        \item A pc-sequence $(a_i)$ is of transcendental type over $K$ if and only if $(P(a_i))$ does not pseudoconverges to $0$, for any nonconstant $P(X)\in K[X]$.
    \end{enumerate}
\end{remark}

\begin{theorem}\label{thm:uniquevaluationpctranscendental}
    Let $(a_i)$ be a pc-sequence in $(K,v)$ of transcendental type over $K$. Let $L = K(X)$ be the field of rational functions over $K$. Then the valuation $v$ extends uniquely to a valuation $v:L\to \Gamma\cup\set{\infty}$ such that 
    \[v(P):= \text{eventual value of $(v(P(a_i)))$}\]
    for each $P\in K[X]$. Further, $(L,v))$ is an immediate valued field extension of $(K,v)$ ($\Gamma_L = \Gamma_K$ and $k_L = k_K$) and $(a_i)\rsa X$. Conversely, if $(a_i)\rsa a$ in a valued field extension of $(K,v)$ then $a$ is transcendental and the field isomorphism $K(X)\to K(a)$ over $K$ sending $X$ to $a$ is a valued field isomorphism.
\end{theorem}
\begin{proof}
    One easily checks that $v(P)$ thus defined is indeed a valuation extending $v$. For instance, if $P = QR\in K[X]\setminus \set{0}$, then let $\alpha,\beta,\gamma $ be the eventual valuations of $(P(a_i))$, $(Q(a_i))$, $(R(a_i))$ respectively. Then for some $i_0 $ we have $P(a_j) = \alpha$, $Q(a_j) = \beta$, $R(a_j) = \gamma$ for all $j>i_0$. As $P(a_j), Q(a_j),R(a_j)$ are elements of $K$ and $P(a_j) = Q(a_j)R(a_j)$ (by the universal property of polynomials) we have $v(P(a_j)) = v(Q(a_j))+v(R(a_j))$ hence $\alpha = \beta+\gamma$, i.e. $v(QR) = v(Q)+v(R)$. We let the other properties of a valuation to check as an exercise. Note that the function $v$ defined on $K[X]$ extends uniquely to $K(X)$ by setting $v(P/Q) = v(P)-v(Q)$. The value group of $(L,v)$ is clearly $\Gamma_K$ as the eventual value of $(P(a_i))$ is the value of $P(a_{j})\in K$ for some big enough $j$. 
    
    We check that $k_L = k_K$. Suppose first that $P\in K[X]$ is such that $v(P) = 0$. As $(a_i)\rsa a$ we have by Theorem \ref{thm:polynomialscontinuous} that $(P(a_i))\rsa P$ so $v(P-P(a_i))$ is eventually strictly increasing. Eventually, $0 = v(P) = v(P(a_i))$ hence $v(P-P(a_i))\geq \min\set{v(P),v(P(a_i)}\geq 0$. As $v(P-P(a_i))$ is eventually strictly increasing, we have $v(P-P(a_i))>0$ eventually, i.e. $\res(P) = \res(b)\in k_K$ for $b=P(a_i)\in \cO_K^\times$. For any $P/Q\in L$ with $v(P/Q) = 0$, as $\Gamma_L = \Gamma_K$ there exists $r\in K$ such that $v(r) = -v(P) = -v(Q)$ so that $P/Q = (rP)/(rQ)$ with $rP,rQ\in K[X]$ with $v(rP) = v(rQ) = 0$. From above, there exists $b_1,b_2\in \cO_K^\times$ with $\res(b_1) = \res(rP)$ and $\res(b_2) = \res(rQ)$. As $v(rQ) = v(b_2)= 0$, $rQ$ and $b_2$ are invertible in $\cO_L^\times$ and $\res((rQ)^{-1}) = \res(b_2)^{-1}$ so that we can apply the ring homomorphism: $\res((rP)/(rQ)) = \res((rP)(rQ)^{-1}) = \res(rP)\res((rQ)^{-1}) = \res(b_1)\res(b_2)^{-1}\in k_K$.

    We check that $(a_i)\rsa X$. By definition, $v(X-a_i)$ is the eventual valuation of $(a_j-a_i)_j$ which is $\alpha_i = v(a_{i+1}-a_i)$ and the sequence $(\alpha_i)$ is eventually strictly increasing by Remark \ref{rk:pc-sequences} (3).

    Finally, assume that $(a_i)\rsa a$ is a valued field extension of $K$. For any $P\in K[X]\setminus K$ we have $(P(a_i))\rsa P(a)$ hence $v(P(a)) = v(P(a_i))$ eventually. By Remark \ref{rk:transcendentaltype12} (2), $v(P(a))\neq \infty $ hence $P(a)\neq 0$ i.e. $a$ is transcendental over $K$. It then follows clearly that the field isomorphism between $K(X)$ and $K(a)$ is a valued field isomorphism.
\end{proof}

The algebraic counterpart is:
\begin{theorem}\label{thm:algebraicpcimmediate}
    Let $(a_i)$ be a pc-sequence in $(K,v)$ of algebraic type over $K$ without pseudolimit in $K$. Then $(K,v)$ admits a proper immediate algebraic extension of valued fields.
\end{theorem}
\begin{proof}
    The proof is very similar to the previous case, a little more complicated. We leave it as an exercise where more details about uniqueness of the immediate extension are given, see Exercise \ref{exo:algebraicpcsequence}.
\end{proof}

\begin{corollary}\label{cor:henselianchar0noalgebraic}
    Let $(K,v)$ be a Henselian valued field of residue characteristic $0$. Let $(a_i)$ be any pc-sequence without pseudo-limit in $K$, then $(a_i)$ is of transcendental type.
\end{corollary}
\begin{proof}
    Otherwise, $(a_i)$ would be of algebraic type and Theorem \ref{thm:algebraicpcimmediate} would contradicts that $(K,v)$ has no proper immediate algebraic extension (Corollary \ref{cor:henselianchar0algmax}).
\end{proof}

\begin{proposition}\label{prop:immediateextensionpseudocauchylimit}
    Let $(K,v)\seq (L,v)$ be a valued field extension which is immediate (i.e. $\Gamma_K = \Gamma_L$ and $k_K = k_L$). Let $a\in L\setminus K$ then there is a limit ordinal $\lambda$ and a sequence $(a_i)_{i<\lambda}$ of elements of $K$ such that 
    \begin{enumerate}
        \item $(a_i)_{i<\lambda}$ has no pseudo-limit in $K$;
        \item $(a_i)\rsa a$.
    \end{enumerate}
    If $(K,v)$ is Henselian of residue characteristic $0$ then $(a_i)$ is of transcendental type.
\end{proposition}
\begin{proof}
    Let $I = \set{v(a-c)\mid c\in K}$. We claim that $I$ has no greater element. If $v(a-c)\in I$, then as $\Gamma_K = \Gamma_L$ there exists $b\in K$ such that $v(a-c) = v(b)$. Then $v((a-c)b^{-1}) = 0$ and as $k_K = k_L$ there exists $d\in \cO_K^\times$ such that $\res((a-c)b^{-1}) = \res(d)$, i.e. $v((a-c)b^{-1} - d)>0$. Then 
    \begin{align*}
        v(a-c - bd) &= v[b(\frac{a-c}{b} - d)]\\
        &= v(b)+\underbrace{v((a-c)b^{-1} - d)}_{>0}\\
        &> v(b) = v(a-c).
    \end{align*}
    For $c' = c-bd\in K$ we have $v(a-c')>v(a-c)$, so $I$ has no greatest element. Choose a sequence $(a_i)_{i<\lambda}$ of element of $K$ such that $\set{v(a-a_i)}$ is strictly increasing and cofinal in $I$, then $(a_i)_{i<\lambda}\rsa a$. By contradiction if there exists $c\in K$ such that $(a_i)\rsa c$, then by Remark \ref{rk:pc-sequences} (1), $v(a-c)>v(a-a_i)$ contradicting cofinality. The last assumption is by Corollary \ref{cor:henselianchar0noalgebraic}.
\end{proof}

\begin{remark}
    Combining Theorem \ref{thm:uniquevaluationpctranscendental}, \ref{thm:algebraicpcimmediate} and Proposition \ref{prop:immediateextensionpseudocauchylimit}, we also obtain the following result of Kaplanski \cite{kaplanskimaximalvalued}: a valued field $(K,v)$ is \textit{maximal} (i.e. it has no proper immediate extension, equivalently any valued field extension extends the residue field or the value group) if and only if every pc-sequence in $(K,v)$ admits a limit in $K$.
\end{remark}

\begin{remark}
    Following on the previous remark, a valued field is called \textit{algebraically maximal} if it admits no proper immediate algebraic extension. Then a valued field $(K,v)$ is algebraically maximal if and only if every pc sequence of algebraic type in $K$ has a limit in $K$. The `if' follows from Proposition \ref{prop:immediateextensionpseudocauchylimit} and Exercise \ref{exo:algebraicpcsequence} and the `only if' follows from Theorem \ref{thm:algebraicpcimmediate}.
    By Corollary \ref{cor:henselianchar0algmax} any Henselian field of residue characteristic $0$ is algebraically maximal and it can be shown that in residue characteristic $0$, being Henselian is equivalent to being algebraically maximal.
\end{remark}

\begin{exercise}\label{exo:differentlimits}
Prove that if $(a_i)\rsa a$ then for all $b$ we have $(a_i)\rsa b$ if and only if $v(a-b)> v(a-a_i)$ eventually.
\end{exercise}

\begin{exercise}\label{exo:algebraicpcsequence}
Let $(a_i)$ be a pc-sequence in $(K,v)$ of algebraic type over $K$ without pseudolimit in $K$. Then $(K,v)$ admits an immediate algebraic extension of valued fields $(L,v)$. Let $P(X)$ be of minimal degree $d$ such that $(v(P(a_i)))$ is eventually strictly increasing.
\begin{enumerate}
    \item Prove that $P$ is irreducible and of degree $\geq 2$.
\end{enumerate}
Let $a$ be a root of $P$ in an extension of $(K,v)$ and let $L = K(a)$. For any polynomial $R(X)\in K[X]$ of degree $<d$, the sequence $v(P(a_i))$ is eventually constant, hence define:
    \[v(R(a)) = \text{eventual value of $v(R(a_i))$}.\]
\begin{enumerate}
\setcounter{enumi}{1}
    \item  Prove that $v$ is a well-defined function on $L^\times$.
    \item Prove that $v$ defines a valuation on $L$ (\textit{Hint}: To prove that $v(S(a)T(a)) = v(S(a))+ v(T(a))$, consider the Euclidean division by $P$: $S(X)T(X) = P(X) Q(X)+R(X)$ with $\deg(R)<d$, then $S(a)T(a) = R(a)$).
    \item Conclude by checking that $\Gamma_L = \Gamma_K$ and $k_L = k_K$ (\textit{Hint.} Proceed as in the proof of Theorem \ref{thm:uniquevaluationpctranscendental}.
    \item (\textit{Bonus}) If $b$ is another root of $P$ in an extension, prove that there is a valued field isomorphism $K(a)\to K(b)$ over $K$ sending $a$ to $b$.
    \end{enumerate}
\end{exercise}

\begin{exercise}\label{exo:pcsequencealgebraictype}
    Assume that $(K(a),v)$ is a proper algebraic immediate extension of $(K,v)$ and $(a_i)$ a pc-sequence of $K$ as in Lemma \ref{prop:immediateextensionpseudocauchylimit} with $(a_i)\rsa a$. Let $P$ be the minimal polynomial of $a$ over $K$.
    \begin{enumerate}
        \item Prove that $P(a_i) = (a_i-a)Q(a_i)$ for some $Q\in K(a)[X]$.
        \item Deduce that $v(P(a_i))$ is eventually strictly increasing, hence that $(a_i)$ is of algebraic type. 
    \end{enumerate}
\end{exercise}


\subsection{The idea of the proof}

Recall that we study ac-valued fields $(K,v,\ac)$ in the language $\LLdp$ of Denef-Pas, consisting of a three sorts: one sort for the field $K$ in a copy $\LLvf$ of the language of rings, one sort for the residue field $k$ in a copy $\LLres$ of the language of rings and one sort for $\Gamma\cup \set{\infty}$ in the language $\LLgp$ of ordered groups expanded by a constant $\infty$. We also have the valuation $v:K\to \Gamma\cup\set{\infty}$ and the angular component map $\ac:K\to k$.

\begin{center} 
\begin{tikzcd}
K \ar[r,"v"] \ar[dr,"\ac"] & \Gamma\cup\set{\infty}\\
\cO_K \arrow[hookrightarrow]{u}  \ar[r,"\res"] & k  
\end{tikzcd}\\
\end{center}

To prove Pas' theorem we will use the following criterion (see e.g. \cite[Lemma 4.2]{Chatzicoursvaluation}).

\begin{lemma}\label{lm:QE}
    Let $T$ be a theory in a countable language $\LL$ and $\Delta$ a set of $\LL$-formulas closed by boolean combination. 
    Let $M$ and $N$ be $\aleph_1$-saturated models of $T$. If the following holds:
    \begin{itemize}
        \item for all $f:A\to B$ isomorphism between two countable substructures $A$ of $M$ and $B$ of $N$ which preserves $\Delta$-formulas ($M\models \phi(a)\implies N\models \phi(f(a))$ for all tuple $a$ from $A$ and $\phi(x)\in \Delta$), then for any $a\in M$ there exists an isomorphism $g$ between two substructures of $M$ and $N$ respectively which extend $f$, preserves $\Delta$-formulas and has the element $a$ in its domain.
    \end{itemize}
    then every $\LL$-formula is equivalent to a $\Delta$-formula modulo $T$. Further, if given any model $M,N$ of $T$, the same $\Delta$-sentences are satisfied by $M$ and $N$, then $T$ is complete.
\end{lemma}

In multi-sorted logic, each sort come equipped with a distinguished set of variable. Quantifying over a variable from a given sort means that the interpretation of the quantification (``for all", ``there exists") is restricted to the elements of the sort. In order to distinguish between the three sorts, we will use different symbols as variable: \begin{itemize}
    \item $x,y,z,\ldots$ for the field sort, 
    \item $\xi, \zeta$ for the group sort,
    \item $\bar x,\bar y,\ldots$ for the residue sort.
\end{itemize}
An example of $\LLdp$ sentence is the following:
\[\forall x \forall \xi  (v(x) = \xi \wedge (\forall \zeta \ \xi+\zeta = \zeta))\to (\exists \bar y\ \ac(x)\bar y = 1).\]
In any ac-valued field, which is a complicated way to express that if $v(x) = 0$ then $\ac(x)$ has an inverse. Note that $\ac(x) = 0$ if and only if $x = 0$ hence 
\[\TDP\models \forall x(x\neq 0 \leftrightarrow \exists \bar y\ \ac(x)\bar y =1)\]
which is a (silly) example of a quantified formula $\exists \bar y\  \ac(x)\bar y =1$ being equivalent modulo $\TDP$ to a quantifier-free one $x\neq 0$. Example \ref{ex:henselianityQE} below gives a less trivial example of quantifier elimination.

We consider the set $\Delta$ of $\LLdp$-formulas in which the quantifiers $\forall,\exists$ only range over variable $\bar x, \bar y$ from the residue field and $\xi,\zeta$ from the value group. The goal of this section is to use Lemma \ref{lm:QE} with the set $\Delta$ to prove that every formula in $\LLdp$ is equivalent modulo $\TDP$ to a formula from $\Delta$. A formula from $\Delta$ will also be called a $\Delta$-formula.

\begin{example}\label{ex:henselianityQE}
    Let us consider now a less trivial example of quantifier elimination in a model $(K,k,\Gamma,v,\ac)$ of $\TDP$. Let $f\in \Z[X]$ be an irreducible polynomial, which can be seen as a polynomial in $K[X]$ and in $k[X]$ because $(K,v)$ is of equicharacteristic $0$. Since $1$ is in the language of rings, 
    \[n = \underbrace{1+\ldots+1}_{n \text{ times}}\]
    is a term hence so are $f(x),f'(x), f(\bar y), f'(\bar y)$. The following $\LLdp$-sentence holds in $(K,v)$
    \[\forall y \left(\underbrace{[\exists x\ f(x) = 0\wedge v(x-y)>0\wedge v(y) = 0]}_{\phi(y)} \leftrightarrow \underbrace{[\exists \bar y \ f(\bar y) = 0\wedge f'(\bar y)\neq 0\wedge v(y) = 0\wedge \ac(y) = \bar y]}_{\psi(y)} \right)\]
    So the formula $\phi(y)$ is equivalent modulo $\TDP$ to the $\Delta$-formula $\psi(y)$ which has no quantifier in the valued field sort.
    We check that the above sentence indeed hold in every model of $\TDP$. As $(K,v)$ is Henselian, we have in particular that every simple root of $f$ in $k$ can be lifted to a root of $f$ in $K$. This gives the right to left direction, the assumption of $v(y) = 0$ is there to ensure that $\res = \ac$. The left to right direction is just applying the ring homomorphism $\res$ to the equation $f(x) = 0$, knowing that $\Z\seq \cO^\times$ and that $f(x)$ is separable.  
\end{example}

In order to apply Lemma \ref{lm:QE}, we consider two $\aleph_1$-saturated models $(K,k_K,\Gamma_K,v,\ac)$ and $(L,k_L,\Gamma_L)$ of $\TDP$. Note that in every model of $\TDP$ the substructure generated by the constants are isomorphic: it is $(\Z,\Z,\set{0,\infty})$ with trivial valuation ($v(a) = 0$ if $a\neq 0$ and $v(0) = \infty$) and $\ac = \Id$. 

We consider two countable substructures $(A,k_A,\Gamma_A)$ and $(B,k_B,\Gamma_B)$ of $(K,k_K,\Gamma_K,v,\ac)$ and $(L,k_L,\Gamma_L)$ respectively and assume that there exists an isomorphism $f: (A,k_A,\Gamma_A)\to (B,k_B,\Gamma_B)$ which preserves formulas in $\Delta$. As an $\LLdp$-isomorphism, $f$ really consists of three maps $f = (f_{\upharpoonright A}, f_{\upharpoonright k_A}, f_{\upharpoonright \Gamma_A})$ where:
\begin{itemize}
    \item $f_{\upharpoonright A}$ is a ring isomorphism between $A$ and $B$;
    \item $f_{\upharpoonright k_A}$ is a ring isomorphism between $k_A$ and $k_B$;
    \item $f_{\upharpoonright \Gamma_A}$ is an ordered group isomorphism between $\Gamma_A$ and $\Gamma_B$, mapping $\infty$ to $\infty$.
    \item $f$ commutes with both $v$ and $\ac$:
    \[f_{\upharpoonright \Gamma_A}(v(a)) = v(f_{\upharpoonright A}(a)) \text{ and } f_{\upharpoonright k_A} (\ac(a)) = \ac(f_{\upharpoonright A}(a)).\]
\end{itemize}

We want prove that for any $a\in K\setminus A$ we may extend $f$ to an $\LLdp$-isomorphism with domain a subset of $(K,k_K,\Gamma_K)$ containing $a$ and which preserves formulas from $\Delta$. We will prove a little more: 
\begin{center}
    for any countable elementary substructure $(E,k_E,\Gamma_E)$ of $(K,k_K,\Gamma_K)$ containing $(A,k_A,\Gamma_A)$, we can extend $f$ to an $\LLdp$-isomorphism with domain $(E,k_E,\Gamma_E)$ which preserves $\Delta$-formulas.
\end{center}
This will give the result since for any $a\in K$ there exists a countable elementary substructure of $(K,k_K,\Gamma_K)$ containing $A$ and $a$, by the downward Lowehein-Skolem theorem (we assume that all languages are countable).

Let $(E,k_E,\Gamma_E)$ be a countable elementary substructure of $(K,k_K,\Gamma_K)$ extending $(A,k_A,\Gamma_A)$. We will extend $f$ to $(E,k_E,\Gamma_E)$ in the following steps:
\begin{itemize}
    \item \textbf{(Step 0)} $k_A$ is an $\LLres$-structure of $k_K$, hence is a ring (or rather an integral domain), we extend $f_{\upharpoonright k_A}$ to the fraction field $\Frac(k_A)$ of $k_A$ and assume that $k_A$ and $k_B$ are fields. We then extend similarly $f_{\upharpoonright A}$ to the fraction field of $A$ to assume that $A, B$ are fields.
    \item \textbf{(Step 1)} We extend $f_{\upharpoonright k_A}$ to $k_E$ to assume that $k_A = k_E$ and similarly we extend $f_{\upharpoonright \Gamma_A}$ to $\Gamma_E$ to assume that $\Gamma_A = \Gamma_E$. This uses that the maps $f_{\upharpoonright k_A}$ and $f_{\upharpoonright \Gamma_A}$ are elementary and $\aleph_1$-saturation of $(L,k_L,\Gamma_L)$.
    \item \textbf{(Step 2)} We extend $f_{\upharpoonright A}$ to a field $C$ with $A\seq C\seq E$ such that $\res: \cO_C\to k_E$ is onto (hence in particular $\ac$ is also onto). To do so, we exhibit a preimage under $\res$ of an element of $k_E$ which is not in the residue $\res(\cO_A)$ of $A$ and extend $f$ to this preimage. There are two subcases: \ref{subsub:stepunramifiedalgebraic} where the element is algebraic over $\res(\cO_A)$ (in which case we use Henselianity of $E$ and Corollary \ref{cor:3.7bestextension} (1)) and \ref{subsub:stepunramifiedtrans} where the element is transcendental over $\res(\cO_A)$ (in which case we use a similar but simpler argument).
    \item \textbf{(Step 3)} We extend $f_{\upharpoonright A}$ to ensure that $v(A^\times) = \Gamma_E$. We exhibit preimage under $v$ of an element in $\Gamma_A\setminus v(A^\times)$ and extend $f$ on it. There are two subcases: \ref{subsub:notorsionmodulo} where the element is of torsion modulo $v(A^\times)$ (in which case we use Henselianity and Corollary \ref{cor:3.7bestextension} (2)) and \ref{subsub:torsionmodulo} where it is not of torsion (in which case we use a similar but simpler argument again).
    \item \textbf{(Step 4)} We extend $f_{\upharpoonright A}$ to $E$. By the previous steps, any element $a\in E\setminus A$ defines an immediate valued field extension $A(a)$ of $A$. There are two subcases: if $a$ is algebraic over $A$, then $a$ belongs to the Henselization $A^h$ of $A$ and the fact that $E$ is henselian and the universal property of Henselization allows to extend $f$ to $A^h$. The other case is when $a$ is transcendental over $A$ in which case Kaplanski theory of pseudo-convergence (subsection \ref{subsub:kaplanski}) yields a unique way of extending $f$ to $A(a)$. 
\end{itemize}

\begin{remark}(Preserving $\Delta$-formulas)\label{rk:preservingdeltaformula}
    We observe that a $\Delta$-formula $\Phi(x,\bar y, \xi)$ is equivalent to a formulas of the form
    \[(Q\bar y') (Q \xi') \psi(x,\bar y, \bar y',\xi, \xi ')\]
    where $x$ is a tuple of $\LLvf$-variables, $\bar y, \bar y'$ are tuples of $\LLres$-variables, $\xi, \xi '$ are tuples of $\LLgp$-variables, $(Q\bar y') (Q \xi')$ are quantifications over those variables and $\psi$ is quantifier-free. As symbols in $\LLdp$ only apply to variables in the appropriate sort, $\psi$ is equivalent to a disjunction of formulas of the form 
    \[\phi_{\mathrm{vf}}(x)\wedge \psi_{\res}(\ac(t_1(x)),\bar y,\bar y')\wedge \psi_{\mathrm{gp}}(v(t_2(x)),\xi, \xi')\]
    where $\phi_{\mathrm{vf}}$ is from $\LLvf$, $\psi_{\res}$ is from $\LLres$, $\psi_{\mathrm{gp}}$ is from $\LLgp$ and $t_1(x), t_2(x)$ are $\LLvf$-terms. This is because the only terms in $\LLgp$ that can be equal to a variable from $\LLgp$ are $\LLgp$-terms and $v(t_1(x))$ for some $\LLvf$-term $t_1$ and similarly for the residue sort. In turn as $(Q\bar y')$ only use free variable from $\psi_{\res}$, we may put $(Q\bar y')$ in front of $\psi_{\res}$ and similarly for $(Q\xi')$ and $\psi_{\mathrm{gp}}$ so that $\psi$ is equivalent to a disjunction of formulas of the form 
    \[\phi_{\mathrm{vf}}(x) \wedge \phi_{\res}(\ac(t_1(x)),\bar y)\wedge \phi_{\mathrm{gp}}(v(t_2(x)),\xi)\]
    for a quantifier-free $\LLvf$-formula $\phi_{\mathrm{vf}}$, an $\LLres$-formula $\phi_{\res}$ and an $\LLgp$-formula $\phi_{\mathrm{gp}}$. In particular,  $f_{\upharpoonright k_A}$ preserves all $\LLres$-formulas, so $f_{\upharpoonright k_A}$ is an elementary (partial) map between $k_A\seq (K,k_K,\Gamma_K)$ and $k_B\seq (L,k_L,\Gamma_L)$, in the following sense:
    \[(K,k_K,\Gamma_K)\models \phi_{\res} (\bar a_1,\ldots,\bar a_n)\iff (L,k_L,\Gamma_L)\models \phi_{\res}(f_{\upharpoonright k_A} (\bar a_1),\ldots,f_{\upharpoonright k_A}(\bar a_n)).\]
    Similarly $f_{\upharpoonright\Gamma_A}$ is an elementary (partial) map $\Gamma_A\to \Gamma_B$.
\end{remark}

Note that $(K,k_K,\Gamma_K)$ and $(L,k_L,\Gamma_L)$ satisfy the same $\Delta$-sentences since those involving the valued field are quantifier-free, hence only talk about the characteristic.

\subsection{Step 0} We extend $f$ to $\Frac(k_A)$ and $\Frac(A)$. 

First, $\Gamma_A$ and $\Gamma_B$ are subgroups of $\Gamma_K$ and $\Gamma_L$ respectively because $\LLgp$ contains $+,-$. The substructure $k_A$ of $k_K$ is a priori an integral domain (a substructure need only be closed under functions of the language, that would be different if there were a function symbol $^{-1}$ for the inverse) and the fraction field $\Frac(k_A)$ is a subset of $k_K$. We extend $f_{\upharpoonright k_A}$ to $\Frac(k_A)$ by setting $f(\frac{\bar a}{\bar b}) = \frac{f(\bar a)}{f(\bar b)}\in K_L$ for $\bar a,\bar b\in k_A$. This extension is unique and $f_{\upharpoonright k_A}$ is still elementary: it is an easy exercise to prove that each $\LLres$-formula involving fractions $ (\frac{\bar a_i}{\bar b_i})$ is equivalent to an $\LLres$-formula involving $a_i,b_i$.  In light of Remark \ref{rk:preservingdeltaformula} the extension of $f$ thus obtained preserves $\Delta$-formulas. $f$ is still an $\LLdp$-isomorphism since there is no commutativity conditions to check. Similarly, the map $f_{\upharpoonright A}$ extends (uniquely) to a field isomorphism between $\Frac(A)$ and $\Frac(B)$ which commutes with $v$: if $\frac{a}{b}\in \Frac(A)$ then $v(\frac{a}{b}) = v(a)-v(b)\in \Gamma_A$ and $f_{\upharpoonright\Gamma_A}(v(a)-v(b)) = v(f_{\upharpoonright A}(a))-v(f_{\upharpoonright B}(b)) = v(\frac{f_{\upharpoonright A}(a)}{f_{\upharpoonright A}(b)}) = v(f(\frac{a}{b}))$. For $a,b\neq 0$ we have $\ac(\frac{a}{b}) = \frac{\ac(a)}{\ac(b)}$ hence because we first extended $f$ to $\Frac(k_A)$ we similarly conclude that $f$ commutes with $\ac$.

\subsection{Step 1} We extend $f$ to $k_E$ and $\Gamma_E$.

Let $(\bar a_i)_{i<\omega}$ be an enumeration of $k_E\setminus k_A$, this exists since $(E,k_E,\Gamma_E)$ is a countable structure. Consider the set $\Sigma(x)$ of $\LLres$-formulas $\phi(\bar x)$ with parameters in $k_A$ and such $(K,k_K,\Gamma_K)\models \phi(\bar a_0)$. Let $\Sigma_0(x)$ be a finite subset of $\Sigma$ and write $\phi(\bar x, \bar c_1,\ldots,\bar c_n)$ for the conjunction of the formulas in $\Sigma_0$, with $\bar c_i\in k_A$. As $(K,k_K,\Gamma_K)\models \psi(\bar a_0,\bar c_1,\ldots,\bar c_n)$ we have $(K,k_K,\Gamma_K)\models \exists x\ \psi(x,\bar c_1,\ldots,\bar c_n)$. Note that $\exists\bar x \phi(\bar x, \bar c_1,\ldots,\bar c_n)$ is a $\Delta$-formula, hence $(L,k_L,\Gamma_L)\models \exists x\ \psi(x,f(c_1),\ldots,f(c_n))$, since $f_{\upharpoonright k_A}$ is elementary. As $\LLdp$ is countable and $(A,k_A,\Gamma_A)$ is countable, the set $\Sigma^f(\bar x) = \set{\phi^f(\bar x)\mid \phi\in \Sigma}$ is also countable, for $\phi^f(\bar x)$ the $\LLres$-formula with parameters in $(B,k_B,\Gamma_B)$ obtained by applying $f$ to every parameters from $k_A$ in $\phi$. By above, $\Sigma^f$ is finitely satisfiable in $(L,k_L,\Gamma_L)$, it follows from $\aleph_1$-saturation that there exists $\bar b_0\in L$ satisfying $\Sigma^f(\bar x)$. It follows from Remark \ref{rk:preservingdeltaformula} and the fact that $\ac(A)\seq k_A$ that the extension of $f$ to $A\cup\set{\bar a_0}\to B\cup\set{\bar b_0}$ by $f(\bar a_0) = \bar b_0$ preserves all $\Delta$-formula. For any $i<\omega$, the set $A\bar a_0,\ldots,\bar a_i$ is still countable hence by induction we may iterate the above argument to extend $f_{k_A}$ to a partial map $k_E\to f(k_E)$ preserving all $\Delta$-formulas. As $k_E$ is a field and $A$ did not change, $(A,k_E,\Gamma_A)$ is still a substructure of $(E,k_E,\Gamma_E)$ and $f$ is an $\LLdp$-isomorphism $(A,k_E,\Gamma_A)\to (B,f(k_E),\Gamma_B)$ which preserves $\Delta$-formulas. 

We follow the exact same strategy for extending $f$ to $\Gamma_E$ by taking an enumeration $(\gamma_i)_{i<\omega}$ of $\Gamma_E\setminus \Gamma_A$, using this time that $f_{\upharpoonright\Gamma_E}$ is $\LLgp$-elementary and that $v(A^\times)\seq \Gamma_A$. In a sense, Remark \ref{rk:preservingdeltaformula} gives a ``separation of sorts", to be understood as: $\tp^{\LLdp}(\gamma_i/(A,k_A,\Gamma_A)\gamma_0\ldots \gamma_{i-1})$ is really given by $\tp^{\LLgp}(\gamma_i/\Gamma_A\gamma_0\ldots \gamma_{i-1})$ . This is really because there is no map going from the group sort (or the residue field sort) to another sort.

\begin{remark}
    In the rest of the proof if we extend $f$ to an $\LLdp$-isomorphism between a substructure $(C,k_E,\Gamma_E)\supseteq (A,k_E,\Gamma_E)$ of $(E,k_E,\Gamma_E)$ and its image, it will automatically preserve all $\Delta$-formulas. This follows from Remark \ref{rk:preservingdeltaformula} since $v(C^\times )\seq \Gamma_E$ and $\ac(C)\seq k_E$ and $f_{\upharpoonright k_E},f_{\upharpoonright \Gamma_E}$ are elementary.
\end{remark}

\textit{By Step 1 and the previous Remark, we are given an $\LLdp$-isomorphism $f: (A,k_E,\Gamma_E)\to (B,f(k_E),f(\Gamma_E))$ which preserves $\Delta$-formulas (in particular $f_{\upharpoonright k_E}$ is $\LLres$-elementary and $f_{\upharpoonright\Gamma_E}$ is $\LLgp$-elementary) and we need to extend it to $(E,k_E,\Gamma_E)$. Note that at this point we may have $\res(\cO_A)\sneq k_E$ and $v(A^\times)\sneq \Gamma_E$. In Step 2 and Step 3 we will ensure that both $\res$ and $v$ are onto.}

\subsection{Step 2} We extend $f$ to a subfield $C\seq E$ such that $\res(\cO_C) = k_E$. 

Denote by $\tilde{k}_A$, $\tilde{k}_B$ the residue fields of $(A,v)$ and $(B,v)$ respectively, i.e. $\res(\cO_A) = \tilde k _A\seq k_A$ and $\res(\cO_B) = \tilde k _B \seq k_B$. As $f$ defines an isomorphism of valued fields between $(A,v)$ and $(B,v)$, it induces an isomorphism between $\tilde k _A$ and $\tilde k _B$. Let $\bar a\in k_E\setminus \tilde k _A$. There are two subcases: $\bar a$ is algebraic over $\tilde k _A$ or $\bar a$ is transcendental over $\tilde k _A$.

\subsubsection{$\bar a$ is algebraic over $\tilde k _A$}\label{subsub:stepunramifiedalgebraic} Let $\bar P(X)$ be its minimal monic polynomial over $\tilde k _A$ and let $P(X)\in \cO_A[X]$ be a monic polynomial obtained by lifting the coefficients of $\bar P(X)$ (and taking $1$ as a lift for the leading coefficient). $P$ has the same degree as $\bar P$ and because $\res: \cO_A\to \tilde k _A$ is a ring homomorphism which extends to $\cO_A[X]\to \tilde k_A[X]$ and $\res(P) = \bar P$ we obtain that $P(X)$ is irreducible over $\cO_A$ and over $A$. As $\tilde k _A$ is of characteristic $0$ and $\bar P$ is irreducible, $\bar P'(a)\neq 0$. As $E$ is Henselian, simple zeros lift, hence there exists $a\in \cO_E$ such that $P(a) = 0$ and $\res(a) = \bar a$. On the other side, $f(P)(X)$ is irreducible over $B$ and of the same degree as $P$ as $f$ is a field isomorphism between $A$ and $B$. Similarly $f(\bar P)(X)$ is irreducible and separable over $\tilde k _B$ with $f(\bar a)$ as a single root, hence by Hensel's Lemma there exists $b\in L$ with $f(P)(b) = 0$ and $\res(b) = f(\bar a)$. We extend $f_{\upharpoonright A}$ to a field isomorphism\footnote{This is very standard: first $A[X]$ and $B[X]$ are isomorphic and the ideal $(P)$ in $A[X]$ (respectively $(f(P))$ in $B[X]$) is the kernel of the evaluation map $A[X]\to A(a)$ (resp. $B[X]\to B(a)$) which yields $A(a)\cong A[X]/(P)\cong B[X]/(f(P))\cong B(b)$.} between $A(a)$ and $B(b)$ by setting $f(a) = b$. We need to check that this isomorphism preserves the valuation and commutes with $\ac$.

As a $\tilde k _A$-vector space, $\tilde k _A (\bar a)$ admits $1, \bar a,\ldots, \bar a^{n-1}$ as a basis. By Corollary \ref{cor:3.7bestextension} (1), for all $u_0,\ldots,u_{n-1}\in A$ we have 
\[v(\sum_{i=0}^{n-1} u_i a^i) = \min_i\set{v(u_i)}\in \Gamma_A\]
Similarly, $1,\res(b), \ldots  \res(b) ^{n-1}$ is a basis of $\tilde k _B (f(\bar a))$ hence $v(\sum_{i=0}^{n-1} f(u_i) b^i) = \min_i\set{v(f(u_i))}\in \Gamma_B$ for all $f(u_i)\in B$. Note that this gives that $A(a)$ and $B(b)$ are unramified extensions of $A$ and $B$ respectively, for the valuations induced on $A(a)$ and $B(b)$ by $(K,v)$ and $(L,v)$. As $f_{\upharpoonright \Gamma_A}$ preserves the order and $v(u_i)\in \Gamma_A$, we have $f(\min_i\set{v(u_i)}) = \min_i\set{f(v(u_i))}$. We also have $f(v(u_i)) = v(f(u_i))$ for all $i$ since $u_i\in A$. We conclude $f(v(\sum_{i=0}^{n-1} u_i a^i)) = v(f(\sum_i u_i a_i))$. 

It remains to check that $f$ commutes with $\ac$. As $A(a)$ is an unramified extension of $A$, by Remark \ref{rk:productofelementsunramified}, every element in $A(a)$ can be written as the product $du$ where $u\in A$ and $d\in A(a)$ with $v(d) = 0$. Then $f(\ac(du)) = f(\ac(d)\ac(u)) = f(\ac(d))f(\ac(u))$. As $f\upharpoonright (A,k_A,\Gamma_A)$ is an $\LLdp$-isomorphism we have $f(\ac(u)) = \ac(f(u))$. As $v(d) = 0$, $\ac(d) = \res(d)$ so $f(\res(d)) = \res(f(d)) = \ac(f(d))$ because $f$ is a valued field isomorphism $A(a)\to A(b)$ (and $v(f(d)) = 0$). We conclude $f(\ac(du)) = \ac(f(du))$.

\subsubsection{$\bar a$ is transcendental over $\tilde k _A$}\label{subsub:stepunramifiedtrans} Then so is $f(\bar a)$ over $\tilde k _B$ and let $a\in \cO_E$, $b\in \cO_L$ be such that $\res(a) = \bar a$ and $\res(b) = \bar b$. As $\res$ is a ring homomorphism we have that $a$ and $b$ are transcendental over $A$ and $B$ respectively. We can extend $f_{\upharpoonright A}$ to a field isomorphism $A(a)\to B(b)$ with $a\mapsto b$. Reasoning as in \ref{subsub:stepunramifiedalgebraic} using Corollary \ref{cor:3.7bestextension} (1) we get that $A(a)/A$ and $B(b)/B$ are unramified valued field extensions (for the induced valuation from $K$ and $L$ respectively) and $f$ commutes with the valuation. Again, as in \ref{subsub:stepunramifiedalgebraic} we get that $f$ also commutes with $\ac$ by writing every element of $A(a)$ as the product of an element of $\cO_{A(a)}^\times$ and an element of $A$ (Remark \ref{rk:productofelementsunramified}).

\textit{By considering a countable enumeration of $k_E\setminus \tilde k _A$ and reasoning as in Step 1, we conclude Step 2.}

\subsection{Step 3} We extend $f_{\upharpoonright A}$ to a subfield $C\seq E$ such that $v(C^\times ) = \Gamma_E$. Let $\tilde \Gamma_A = v(A^\times)$ be the value group of $A$ and let $\gamma\in \Gamma_A\setminus \tilde \Gamma_A$ with $\gamma>0$. Let $\tilde \Gamma_B =v(B^\times) = f(\tilde \Gamma_A)$. There are two cases: there exists $n$ such that $n\gamma\in \tilde\Gamma_A$ ($\gamma$ is torsion modulo $\tilde \Gamma_A$) or $n\gamma\notin \tilde\Gamma_A$ for all $n$ ($\gamma$ is not torsion modulo $\tilde\Gamma_A$). Note that at this point $(A,v)$ and $(E,v)$ have same residue field by Step 2.

\subsubsection{$\gamma$ is torsion modulo $\tilde \Gamma_A$}\label{subsub:torsionmodulo} Let $n$ be minimal such that $n\gamma\in \tilde \Gamma _A$. Observe that $0, \gamma,\ldots, (n-1)\gamma$ are in different cosets modulo $\tilde \Gamma _A$. 
As $(A,v)\seq (E,v)$ have same residue field and $(E,v)$ is Henselian, by Corollary \ref{cor:henseliannthroots}
there exists $a\in E$ such that $a^n\in A$ and $v(a) = \gamma$. By minimality of $n$ we also have that $X^n-a^n$ is the minimal polynomial of $a$ over $A$: otherwise $\sum_{i=0}^{n-1} u_i a^i = 0$ for some $u_i\in A$ not all zero, then by Corollary \ref{cor:3.7bestextension} (2), $v(\sum_{i=0}^n u_i a^i) = \min_i\set{v(u_i)+i\gamma}$ so this implies that $v(u_i)+i\gamma = v(u_j)+j\gamma $ for some $0\leq i<j<n$, contradicting minimality of $n$.

As $f_{\upharpoonright\Gamma_A}:\Gamma_A\to \Gamma_B$ is an isomorphism and $f_{\upharpoonright \Gamma_A}(\tilde \Gamma_A)=\tilde \Gamma_B$, $n$ is also minimal such that $nf(\gamma)\in \tilde\Gamma _B $ and $0,f(\gamma),\ldots,(n-1)f(\gamma)$ are in different cosets modulo $\tilde \Gamma _B$. By Fact \ref{fact:henselization}, the Henselization $B^h$ of $B$ is a subfield of $(L,v)$ with same residue field. By Corollary \ref{cor:henseliannthroots}, this time applied to the valued field extension $(B^h,w)$ of $(B,w)$, there exists $c\in B^h\seq L$ such that $c^n\in B$ and $v(c) = f(\gamma)$. Again by minimality of $n$, $X^n-c^n$ is the minimal polynomial of $c$ over $B$. At this point, if it was the case that $c^n = f(a^n)$, it would follow that $f$ extends to a field isomorphism $A(a)\to B(c)$ via $a\mapsto c$ and by Corollary \ref{cor:3.7bestextension} (2) $v(\sum_{i=0}^n u_i a^i) = \min_i\set{v(u_i)+iv(a)}$ and $v(\sum_{i=0}^n f(u_i) b^i) = \min_i\set{f(v(u_i))+iv(b)}$ for all $u_i\in A$, which only depends on $\tilde\Gamma_A,\gamma$ (respectively $\tilde \Gamma_B, f(\gamma)$) so, similarly to the previous case, $f$ is a valued field isomorphism. In order to commute with $\ac$ we need to modify $c$ to some $b$ such that $\ac(b) = f(\ac(a))$.

We change $c$ to $b$ such that $f(a^n) = b^n$ and $\ac(b) = f(\ac(a))$. As $c^n\in B$ the element $\ac(c)\in k_L$ is algebraic over $f_{\upharpoonright k_A}(k_A) = k_B$. As $f_{\upharpoonright k_A}$ is $\LLres$-elementary, we have $k_B\prec k_L$ which implies\footnote{This is pretty straightforward: let $K\prec L$ as rings and $a\in L$ algebraic over $K$. Let $P(X)\in K[X]\setminus \set{0}$ be the minimal monic polynomial of $a$ over $K$ and $n$ the number of roots of $P$ in $L$. Then $L\models \exists ^{\geq n} x\ P(x) = 0$ hence $K\models \exists ^{\geq n} x\ P(x) = 0$ so that every root of $P$ in $L$ is also in $K$, in particular $a\in K$. Here $\exists ^{\geq n} x P(x) = 0$ is a shortcut for $\exists x_1,\ldots,x_n \bigwedge_{1\leq i\neq j\leq n} x_i\neq x_j\wedge \bigwedge_{1\leq i\leq n} P(x_i) = 0$.} $k_B^\alg\cap k_L = k_B$, so $\ac(c)\in k_B$. Consider $f(\ac(a))\ac(c^{-1})\in k_B\setminus\set{0}$. By Step 2, $\res(\cO_B) = f(k_A) = k_B$ hence there exists $d\in \cO_B^\times $ such that $\res(d) = f(\ac(a))\ac(c^{-1})$. Let $c' = cd$ then $\ac(c') = f(\ac(a))$ and $v(c') = v(c) = f(\gamma)$. It follows that $v(f(a^n)/c'^n) = 0$ hence 
\[\res(f(a^n)/c'^n) = \ac(f(a^n))\ac(c^n)^{-1} = 1.\]
Then, $f(a^n) = c^n(1+u)$ for some $u\in \MM_B$ and by Claim \ref{claim_crux} (applied again in the immediate extension $B^h\seq L$) we have that there exists $d\in L$ with $e^n = 1+u$ and $\res(e) = 1$. We set $b = c'e$, so that $b^n = f(a^n)$ and $\ac(b) = \ac(c') = f(\ac(a))$.


Note that Corollary \ref{cor:3.7bestextension} (2) implies that $v(A(a)^\times) = \vect{\Gamma_A,\gamma}$ hence by Remark \ref{rk:productofthreeelementsramifiedcase} every element of $A(a)$ can be written as product of elements $uca^n$ with $u\in A, c\in A(a)$ with $v(c) = 0$ and similarly on the side of $B(b)$. Then using $\ac(b) = f(\ac(a))$ and the fact that $\ac$ and $\res$ coincide on elements of valuation $0$ we conclude that $f:A(a)\to B(b)$ commutes with $\ac$.

\subsubsection{$\gamma$ is torsion-free modulo $\tilde \Gamma _A$}\label{subsub:notorsionmodulo} Let $a\in E$ be such that $v(a) = \gamma$. As $(n\gamma)_{n\in \N}$ are in different cosets modulo $\tilde \Gamma _A$, Corollary \ref{cor:3.7bestextension} (2) yields \[v(\sum_i u_i a^i) = \min_i \set{v(u_i)+iv(a)}\quad (\star)\] for all $(u_i)\in A$. $(\star)$ has several consequences:
\begin{itemize}
    \item $a$ is transcendental over $A$: if $\sum_i u_i a^i = 0$, this would imply that $v(u_i)+iv(a) = v(u_j)+jv(a)$ for some $i\neq j$ contradicting the hypothesis on $\gamma$. 
    \item If $b\in L$ is such that $v(b) = f(\gamma)$ then as in the previous point, $b$ is transcendental over $B$ and by $(\star)$ the field isomorphism $A(a)\to B(b)$ mapping $a\mapsto b$ commutes with the valuation
    \item $v(A(a)^\times)  = \tilde \Gamma _A\oplus \vect{\gamma}$ and $v(B(b)^\times ) = \tilde \Gamma _B\oplus \vect{f(\gamma)}$
\end{itemize}
We prove that we may choose $a,b$ as above with $\ac(a) = 1$. As $0<\alpha<\infty$ we have that $\res(a) = 0$ and as $a\neq 0$ we have $\ac(a)\neq 0$. In particular there exists $u\in \cO_A^\times $ such that $\res(u) = \ac(a)$ hence for $a' = a u^{-1}$ we have: $a'$ transcendental over $A$, $v(a') = v(a) - v(u) = \gamma$ and $\ac(a') = 1$. Similarly we may move $b$ to $b'$ such that $\ac(b') = 1$. By above the extension $f: A(a')\to B(b')$ mapping $a'\mapsto b'$ is a valued fields isomorphism and preserves $\ac$ using Remark \ref{rk:productofthreeelementsramifiedcase}.

\textit{By considering a countable enumeration of $\Gamma_E\setminus \tilde \Gamma _A$ and reasoning as in Step 1, we conclude Step 3.}

\subsection{Step 4} We extend $f_{\upharpoonright A}$ to $E$. For any $a\in E\setminus A$, we have $v(a)\in \Gamma_E = v(A^\times) = \Gamma_A$ by Step 3 and if $v(a)\geq 0$ we have $\res(a) \in k_E = \res(\cO^\times_A) = k_A$ by Step 2 so $A(a)$ is an immediate extension of $A$. If $v(a)<0$ consider $a^{-1}$ (since we will extend $f$ to the field $A(a)$). There are two subcases.

\subsubsection{$a$ is algebraic over $A$} Then by Fact \ref{fact:henselization} (4) we have that $a$ is in the Henselization $A^h$ of $A$. We extend directly $f$ to $A^h$. Note first that as $E$ is Henselian, $A^h\seq E$ and $A^h = A^\alg\cap E$ by Fact \ref{fact:henselization} and the restriction of the valuation of $E$ to $A^h$ is the (unique) extension of $v_A$ to $A^h$. Similarly, $B^h$ is an immediate extension of $B$ and $f$ extends (uniquely) to a valued field isomorphism $A^h\to B^h$ this follows from the universal property of the Henselization (Fact \ref{fact:henselization} (3)). Note that the Henselization is an immediate extension hence $\Gamma_A (=\Gamma_E)$ $k_A (=k_E)$, $\Gamma_B$, $k_B$ are left unchanged. As $A^h$ is an unramified extension of $A$, $\ac$ extends uniquely to $A(a)$ (by Remark \ref{rk:productofelementsunramified}) and $f$ commutes with $\ac$.

\subsubsection{$a$ is transcendental over $A$} In this case, by Proposition \ref{prop:immediateextensionpseudocauchylimit}, as $A$ is Henselian of residue characteristic $0$, there exists a pc-sequence of transcendental type $(a_i)_{i<\lambda}$ of elements of $A$ which has no limit in $A$ and such that $(a_i)\rsa a$. Then $(A(a),v)$ (in $E$) is the unique extension of $(A,v)$ as in Theorem \ref{thm:uniquevaluationpctranscendental}. As $(a_i)_{i<\lambda}$ is a pc-sequence of transcendental type, so is the sequence $(b_i)_{i<\lambda}$ defined by $b_i = f_{\upharpoonright A}(a_i)$, since being a pc-sequence of transcendental type is expressible as a set of quantifier-free formulas\footnote{Eventually we have $v(P(a_i))_i$ is constant for all $P\in A[X]\setminus \set{0}$ so this holds for $(P(b_i))_i$ eventually for all $P\in B[X]\setminus \set{0}$.}. By Lemma \ref{lm:everypcsequenceisconvergent} $(b_i)$ has a pseudolimit $b$ in $L$ and $f$ extends to a valued field isomorphism $A(a)\to B(b)$ mapping $a\mapsto b$ by the ``Conversely" part of Theorem \ref{thm:uniquevaluationpctranscendental}. As in the previous case, $A(a)$ is an unramified extension of $A$ so $\ac$ extends uniquely to $A(a)$ (by Remark \ref{rk:productofelementsunramified}) and $f$ commutes with $\ac$.

\textit{By considering a countable enumeration of $E\setminus A$ and reasoning as in Step 1, we conclude Step 4.}

\appendix

\section{Extra results of Ax and Kochen and a conjecture of Lang}

Recall that a ring $R$ is \textit{local} if it has a unique maximal ideal $\cM$. The field $k:= R/\cM$ is called the \textit{residue field}, just as in the valued field case. Note that $R^\times = R\setminus \cM$. A local ring $R$ is \textit{Henselian} if for all polynomials $P(X)\in R[X]$ and any $a\in R$ such that 
\[P(a)\in \cM\quad \text{and}\quad P'(a)\notin \cM\]
there is $b\in R$ such that $P(b) = 0$ and $a-b\in \cM$.

\begin{theorem}[Lifting]\label{thm:lifting}
    Let $R$ be a Henselian local ring of residual characteristic $0$. Then the residue field can be lifted, i.e. there is a field $F$ which is a subring of $R$ such that $\res : F\to k$ is an isomorphism. 
\end{theorem}
\begin{proof}
    As $R$ is of residual characteristic $0$, $k$ contains $\Q$ as a subfield and the injective homomorphism $\Z\to k$ lift to a homomorphism $\Q\to k$. In particular $R$ contains the ring $\Q$ as a subring. Using Zorn's Lemma, there exists a maximal field $F$ contained in $R$ extending $\Q$. Note that $F$ consist of invertible elements of $R$ hence $\res$ is injective on $F$, so $\res(F)$ is a subfield of $k$. We now show that $F$ is a lift of $k$, in the sense that $\res(F)=k$. Assume not, then there exists $\bar a\in k$ such that $\bar a\notin \res(F)$. There are two cases. First if $\bar a$ is transcendental over $\res(F)$. Then for all $P\in F[X]\setminus \set{0}$ and any lift $a$ of $\bar a$ ($\res(a) = \bar a$) we have 
    \[\res(P(a)) = \res(P)(\bar a)\neq 0.\] 
    This means that $P(a)$ is invertible in $R$, i.e. $P(a)\in R^\times$ so that for any $a\in \res^{-1}(\bar a)$ we get that $F(a)$ is a field contained in $R$. This contradicts maximality of $F$. If $\bar a$ is algebraic over $\res(F)$ hence there exists a monic polynomial $P(X)\in F[X]$ which is a lift of the minimal monic polynomial $\res(P)(X)$ of $\bar a$ over $\res(F)$. Then $P(X)$ is irreducible over $F[X]$. As $k$ is of characteristic $0$ we have $\res(P(a)) = 0$ and $\res(P'(a)) = (\res(P))'(\bar a) \neq 0$ which means $P(a) \in \cM$ and $P'(a)\notin\cM$. As $R$ is Henselian, there exists $b\in R$ such that $P(b) = 0$ and $\res(b) = \res(a) = \bar a$. This means that $F[b] = F(b)$ is an algebraic field extension of $F$ contained in $R$, contradicting maximality of $F$. We conclude that such $\bar a$ cannot exist, i.e. $\res:F\to k$ is onto.
\end{proof}

A consequence of the lifting theorem is the following extra result of Ax and Kochen \cite{AK65I}.
\begin{theorem}[Ax-Kochen]\label{thm:extraresultaxkoch}
    Let $P\in \Z[X_1,\ldots,X_n]\setminus \set{0}$ and $\bar P\in \F_p[X_1,\ldots,X_n]$ its reduction modulo $p$. Then for all but finitely many primes $p$, any zero of $\bar{P}$
    in $\F_p$ can be lifted to a zero of $P$ in $\Z_p$.
\end{theorem}

\begin{proof}
    Let $P\in \Z[X_1,\ldots,X_n]$. Let $\cU$ be a non-principal ultrafilter on the set of prime numbers. Reasoning as in the proof of the Ax-Kochen principle, $R := \prod_\cU\Z_p$ is a Henselian local field with residue field $k :=\prod_\cU \F_p$ hence of characteristic $0$. Assume that $P$ has a zero in $k$, i.e. there exists $\bar a\in k^n$ such that $\res(P)(\bar a)= 0$. By Theorem \ref{thm:lifting}, $P$ considered in $R[X]$ also have a solution in $R$, i.e. there exists $b\in R^n$ such that $\res(b_i) = \bar a_i$ and $P(b) = 0$, by lifting $\bar a$. Note that here both $R$ and $k$ have characteristic $0$. Consider now the statement $\phi_P$ defined by
    \[\forall x\  \left( P(x)\in\cM\right)\rightarrow \left(\exists y P(y) = 0\wedge y_i-x_i\in \cM\right) \]
    where $x = (x_1,\ldots,x_n)$ and $y= (y_1\ldots y_n)$. Then $R\models \phi_P$ hence by \L o\'s theorem, we get $\Z_p\models \phi_P$ for all but finitely many primes $p$.
\end{proof}

As a corollary, Ax and Kochen \cite{AK65I} were able to establish a solution to a conjecture of Lang from the 50s, which were already proven by Greenleaf using algebraico-geometric technics around the same time.

\begin{corollary}(Greenleaf, Ax-Kochen)
    Let $P\in \Z[X_1,\ldots,X_n]\setminus \set{0}$ be with constant term equal to zero and assume that $\deg P<n$. Then $P$ has a nontrivial zero in $\Z_p$ for all but finitely many primes $p$.
\end{corollary}
\begin{proof}
    The polynomial $\res(P)(X)$ has a trivial solution in $\F_p$ since there is no constant terms. By the Chevalley-Warning theorem, the cardinality of the set
    \[\set{a\in \F_p^n\mid P(a) = 0}\]
    is divisible by $p$ hence there is a nontrivial zero of $\res(P)$ in $\F_p$. By Theorem \ref{thm:extraresultaxkoch}, this nontrivial zero lift to a zero in $\Z_p$ (which is nontrivial) for all but finitely many primes.
\end{proof}

The lifting theorem has another consequence. Let $n\in \N^{>1}$ and $p$ a prime number. We list similarities and differences between the two following rings.~\\

\begin{minipage}{0.45\textwidth}
\[\Z/p^n\Z = \set{0,1,\ldots, p^{n}-1}\]
~\\
\begin{itemize}
    \item Ring of cardinality $p^n$.
    \item Local ring with residue field $\F_p$.
    \item The unique maximal ideal $\vect{p}$ is principal.
    \item The characteristic is $p^n$.
    \item No lift of the residue field.
\end{itemize}
\end{minipage}%
\hfill
 \vrule height 0.1\textheight depth 0.1\textheight
\hfill
\begin{minipage}{0.45\textwidth}
\[\F_p[X]/\vect{X^n} \]
~\\
\begin{itemize}
    \item Ring of cardinality $p^n$.
    \item Local ring with residue field $\F_p$.
    \item The unique maximal ideal $\vect{X}$ is principal.
    \item The characteristic is $p$.
    \item Lift of the residue field $\F_p$.
\end{itemize}
\end{minipage}
~\\

Thus, the two rings $\Z/p^n\Z$ and $\F_p[X]/\vect{X^n}$ look very much alike. As for $\Z_p$ and $\F_p[[X]]$ in the AKE theorem, we show that $\Z/p^n\Z$ and $\F_p[X]/\vect{X^n}$ asymptotically share the same first-order theory. We will study local rings in the language $\LLr^t = \LLr\cup\set{t}$ where $t$ is a constant symbol. For $n\in \N^{>1}$, let $T_n$ be the $\LLr^t$-theory of rings $R$ such that
\begin{enumerate}
    \item $R$ is a local ring,
    \item the maximal ideal is principal, generated by the constant $t$: $\cM = \vect{t}$,
    \item the residue field $k = R/\cM$ is of characteristic $0$,
    \item $t^{n-1} \neq 0$ and $t^n = 0$.
\end{enumerate}

We write $(R,k,t)$ for a model of $T_n$ so that $R$ is the local ring, $k = R/\cM$ is the residue field and $t$ is the generator of $\cM$. 

\begin{theorem}\label{thm:extraAKE}
    Let $(R,t)$ be a model of $T_n$ with residue field $k$. Then 
    \[(R,t) \cong (k[X]/\vect{X^n},X)\]
    In particular for two models $(R,k,t)$ and $(R',k',t')$ of $T_n$ we have 
    \[k\cong k'\iff (R,k,t)\cong (R',k',t').\]
\end{theorem}

\begin{proof} We start with a claim.
\begin{claim}\label{claim:last}
    Let $R$ be a local ring with $t$ a generator of the maximal ideal $\cM$ of $R$. Let $A$ be a set of representatives of $R$ modulo $\cM$. Then:
    \begin{enumerate}
        \item for each $n\in \N$, $r\in R$, there exists $a_0,\ldots, a_{n-1}\in A$ and $s\in R$ such that
        \[r = a_0+a_1t+\ldots+a_{n-1}t^{n-1}+st^n.\]
        \item If $t^{n-1}\neq 0$ then $t^m\notin \vect{t^{m+1}}$ for $m< n$ and $(a_0,\ldots,a_{n-1})$ in $(1)$ is uniquely determined by $n,r$.
    \end{enumerate}
\end{claim}
\begin{proof}[Proof of Claim \ref{claim:last}]
    See Exercise \ref{exo:appendix}.
\end{proof}
Using $(2)$ of the claim, there is a strictly descending chain of ideals
\[R = \vect{1}\supsetneq \vect{t}\supsetneq \ldots \supsetneq \vect{t^{n-1}}\supsetneq \vect{t^n} = \set 0.\]
Hence we may define a valuation $v:R\to \Z\cup\set{\infty}$ by 
\[v(r) = \begin{cases}
    \max\set{i\mid r\in \vect{t^i}} &\text{if $r\neq 0$}\\
    \infty &\text{if $r = 0$}
\end{cases}\]
and a norm on $R$ via $\abs{r} = 2^{-v(r)}$. Since $v$ takes only finitely many values, $R$ is complete in the metric induced by $\mid\cdot\mid$ so $R$ is Henselian (see Remark \ref{rk:henselslemma}). By Theorem \ref{thm:lifting}, there is a ring embedding $j:k\to R$ such that $\res(j(x)) = x$ for all $x\in k$. In particular, $A := j(k)$ is a set of representatives modulo $\cM$. By the universal property of polynomial rings, $j$ extends to a ring homomorphism $j_t:k[X]\to R$ by setting $j_t(X) = t$. By $(1)$ of the claim, $j_t$ is onto and by $(2)$ $\ker j_t = \vect{X^n}$.
\end{proof}

Here is an immediate consequence, which could be considered as and Ax-Kochen principle for finite local rings.

\begin{corollary}
    Let $\cU$ be a non-principal ultrafilter on the set of prime numbers. Then for any $n>1$ 
    \[\prod_\cU (\Z/p^n\Z, p)\cong \prod_\cU (\F_p[X]/\vect{X^n}, X) \quad\text{as $\LLr^t$-structure}.\]
    In particular, for any $\LLr$-sentence $\phi$ 
    \[(\Z/p^n\Z, p)\models \phi \iff (\F_p[X]/\vect{X^n}, X)\models \phi\]
    for all but finitely many primes $p$.
\end{corollary}
\begin{proof}
    Reasoning as in the proof of the Ax-Kochen principle, the residue field on both sides is the pseudo-finite field $\prod_\cU \F_p$ which is of characteristic $0$. It follows that $\prod_\cU (\Z/p^n\Z, p)$ and $\prod_\cU (\F_p[X]/\vect{X^n}, X)$ are both models of $T_n$ with the same residue field, so Theorem \ref{thm:extraAKE} applies.
\end{proof}

\begin{remark}[For model-theorists]
    One could ask the following question about the rings $ \prod_\cU \Z/p^n\Z$ and $\prod_\cU \F_p[X]/\vect{X^n}$: where do they lie in Shelah's classification landscape? Using Theorem \ref{thm:extraAKE} we have that $R \cong F[X]/\vect{X^n}$ for $F = \prod_\cU \F_p$. It is easy to see that $F[X]/\vect{X^n}$ is definable\footnote{Consider $A = F^n$, defines the addition componentwise and the following multiplication:
    \[(a_0,\ldots,a_{n-1})*(b_0,\ldots,b_{n-1}) := (a_0b_0,\ldots,\sum_{i+j = k} a_ib_j, \ldots, \sum_{i+j = {n-1}} a_ib_j).
    \] Then $(A,+,*)\cong (F[X]/\vect{X^n},+,\cdot)$.} in the pure theory of the pseudo-finite field $F$, so the rings $ \prod_\cU \Z/p^n\Z$ and $\prod_\cU \F_p[X]/\vect{X^n}$ are simple.
\end{remark}

\begin{exercise}\label{exo:appendix}
    The goal here is to prove Claim \ref{claim:last}. We keep the same notations as in the claim.
    \begin{enumerate}[label=(\alph*)]
        \item Prove $(1)$ using induction on $n$.
        \item Prove that if $t^{n-1}\neq 0$ then $t^m\notin \vect{t^{m+1}}$ for all $m<n$. (\textit{Hint.} Observe that $1-tr$ is a unit, for all $r\in R$.)
        \item Assume that $\sum_{i = 0}^{n-1} a_i ^i + rt^{n} = \sum_{i = 0}^{n-1} b_i ^i + st^{n} $ for $a_i,b_i\in A$, $r,s\in R$. By contradiction, let $m$ be the least $i$ such that $a_i\neq b_i$. Prove that $(a_m-b_m)t^m\in \vect{t^{m+1}}$.
        \item Prove that $a_m-b_m$ is invertible.
        \item Conclude using $(b)$.
\end{enumerate}
\end{exercise}

\begin{exercise}
    Let $T_n^\ACF$ be the expansion of $T_n$ expressing further that the residue field is algebraically closed. The goal of this exercise is to prove that the theory $T_n^\ACF$ is complete and axiomatize $(\C[X]/\vect{X^n},X)$.
    \begin{enumerate}[label=(\alph*)]
        \item Check that $T_n^\ACF$ is indeed first-order.
        \item Let $R$ be a model of $T^\ACF_n$, prove that $\abs{R} = \abs{k}$ where $k$ is the residue field. (\textit{Hint.} Observe that $R$ is isomorphic to a $k$-vector space of dimension $n$).
        \item Let $\kappa$ be an uncountable cardinal. Prove that two models of cardinality $\kappa$ of $T_n^\ACF$ are isomorphic. $T_n^\ACF$ is called \textit{uncountably categorical}. (\textit{Hint.} Use that two algebraically closed fields of uncountable cardinality are isomorphic and Theorem \ref{thm:extraAKE}.)
        \item Conclude that $T_n^\ACF$ is complete. (\textit{Hint.} Take two arbitrary models $R,R'$ of $T_n^\ACF$ and consider elementary extensions of uncountable cardinality.)
        \item As an application, prove the following Lefschetz principle: for all $\LLr^t$-sentence $\phi$
        \[(\C[X]/\vect{X^n},X)\models \phi \iff (\F_p^\alg [X]/\vect{X^n},X)\models \phi\]
        for all but finitely many primes $p$.
    \end{enumerate}
\end{exercise}

\bibliographystyle{plain}
\bibliography{biblio}

\end{document}